\documentclass[a4paper,11pt]{amsart}
 \usepackage[arrow,matrix]{xy}
\usepackage{amsmath,amssymb,amscd,bbm,amsthm,mathrsfs, chemarrow}
\usepackage{amsmath}

\usepackage{empheq}

\usepackage{xcolor}

\usepackage{graphicx}

\definecolor{hellmagenta}{rgb}{1,0.75,0.9}

\definecolor{hellcyan}{rgb}{0.75,1,0.9}

\definecolor{hellgelb}{rgb}{1,1,0.8}

\definecolor{colKeys}{rgb}{0,0,1}

\definecolor{colIdentifier}{rgb}{0,0,0}

\definecolor{colComments}{rgb}{1,0,0}

\definecolor{colString}{rgb}{0,0.5,0}

\definecolor{darkyellow}{rgb}{1,0.9,0}

 \numberwithin{equation}{section}
\newtheorem{thm}{Theorem}[section]
\newtheorem{lem}{Lemma}[section]
\newtheorem{cor}{Corollary}[section]
\newtheorem{pro}{Proposition}[section]

\theoremstyle{definition}

\theoremstyle{remark}

\newtheorem{rem}{Remark}[section]
\begin{document}\numberwithin{equation}{section}
\title[Spectral projection operators of the Sub-Laplacian]
{Spectral projection operators of the Sub-Laplacian and Laguerre calculus on non-degenerate nilpotent Lie groups of step two}
\author{Qianqian Kang, Der-Chen Chang  and Wei Wang }
\thanks{The first author is  partially supported by National Nature Science Foundation in China (No. 11801523) }
         \thanks{The second author is partially supported by a McDevitt Endowment Fund at Georgetown University.}
          \thanks{The third author is partially supported by National Nature Science Foundation in China (No. 12371082).}

\address{Department of Mathematics and Statistics, Georgetown University, Washington D.C. 20057, USA
 Graduate Institute of Business Administration, College of Management, Fu Jen Catholic University, Taipei 242, Taiwan, ROC}
\email{chang@georgetown.edu}

  \address{ Department of Mathematics,  Zhejiang International Studies University,  Hangzhou 310023, PR China}
\email{qqkang@zisu.edu.cn}

\address{Department of Mathematics, Zhejiang University, Zhejiang 310027, PR China}
\email{wwang@zju.edu.cn.}

 \begin{abstract}
In this paper, we  introduce the spectral projection operators $\mathbb{P}_m$ on non-degenerate nilpotent Lie groups $\mathcal{N}$ of step two,  associated to the joint spectrum of sub-Laplacian and derivatives in step two. We construct their kernels $P_m(\mathbf{y},\mathbf{t})$ by using  Laguerre calculus and find a simple integral representation formula for $\mathbf{y}\neq 0$.  Then we show the kernels are Lipschitzian  homogeneous functions on $\mathcal{N}\setminus \{\mathbf{0}\}$ by analytic continuation. Moreover, they are shown to be Calder\'{o}n-Zygmund kernels, so that the spectral  projection operator $\mathbb{P}_m$ can be extended to a bounded operator from $L^p(\mathcal{N})$ to itself. We also prove a convergence theorem of the Abel sum $\lim _{R \rightarrow 1^-} \sum_{m=0}^{\infty} R^{m}\mathbb{P}_{m}\phi=\phi$ by  estimating the $L^p(\mathcal{N})$-norms of  $\mathbb{P}_m$. Furthermore, $\mathbb{P}_m$ are mutually orthogonal projection operators and $\sum_{m=0}^{\infty} \mathbb{P}_{m}\phi=\phi$ for $\phi\in L^2(\mathcal{N})$.
  \end{abstract}

\maketitle

\section{Introduction}
The Heisenberg groups $\mathcal{H}_n$ are non-commutative nilpotent Lie groups, with underlying manifold $\Bbb R^{2n+1}$ and the group law
\begin{equation*}
   (\mathbf{x},t)\cdot (\mathbf{x}^{\prime},s)=\bigg(\mathbf{x}+\mathbf{x}^{\prime},t+s-2 \sum_{j=1}^{n}a_j(x_jx_{n+j}^{\prime}-x_{n+j}x_{j}^{\prime})\bigg),
\end{equation*}
where $\mathbf{x}=(x_1,\ldots,x_{2n}), \mathbf{x}^{\prime}=(x^{\prime}_1,\ldots,x^{\prime}_{2n})\in \Bbb R^{2n}$, $a_j\in \Bbb R_{+}, j=1,2,\ldots,n$. Many aspects of harmonic analysis on the Euclidean space have been successfully generalized to the Heisenberg group (cf. \cite{Stein}). But the theory of Fourier transform on the Heisenberg group has not been completed so far (cf. \cite{FS}\cite{Ge}\cite{G1}\cite{G2}\cite{LCC} and reference therein).
The spectral theory of differential operators on the Heisenberg group is also more difficult and complicated than its  Euclidean counterpart.  \par Strichartz \cite{S}  investigated the joint spectral
theory of the sub-Laplacian operators $\Delta_b$ and $iT$ on the Heisenberg group, where  $T=\frac{\partial}{\partial t}$.
Since $\Delta_b$ and $iT$ are essentially self-adjoint strongly commutative operators, there is a well defined joint spectrum, which is the complement of the set of $(\lambda, \mu)\in \Bbb C^2$ where neither $\lambda I-\Delta_b$ or $\mu I-iT$ is invertible. It is the closed subset of the plane called {\it Heisenberg fan}, consisting of the union of rays
\begin{equation*}
  R_{k,\pm 1}=\left\{(\lambda,\mu):  \mu=\frac{\pm \lambda}{n+2k}\right\}
\end{equation*}
for $k\in \mathbb{N}_0$ and the limit ray $R_{\infty}=\{(\lambda,\mu):  \mu=0, \lambda\geq 0\}.$
The associated spectral projection  operators are convolution operators with kernel that can be described explicitly. These kernels are homogeneous functions of degree $-2n-2$ on $\mathcal{H}_{n}\setminus \{\mathbf{0}\}$ and are Calder\'{o}n-Zygmund kernels. Strichartz also gave the spectral decomposition for all $\phi\in L^p(\mathcal{H}_{n})$, $1<p<\infty$.
  Chang-Tie \cite{CT} used a different approach to this problem from the point of view of Laguerre calculus and derived the explicit kernels for the projection operators. They deduced the regularity properties for spectral  projection operators in other function spaces, e.g. Hardy space as well as Morry space.

Recently, there are many works for harmonic analysis on nilpotent Lie group of step two and its applications to other parts of analysis (cf. \cite{CDW}\cite{Pe}\cite{PR}\cite{SW}\cite{wang5}\cite{ww}\cite{xinfengwu} and references therein). For example, as the first spectral projection operator on the Heisenberg group is the distinguished Szeg\"{o} projection, the first spectral projection operator on some nondegenerate nilpotent Lie group of step two is associated to the Szeg\"{o} projection in quaternionic analysis(cf. \cite{SW}\cite{wang5}\cite{ww}). It is necessary to develop the joint spectral theory on this kind of groups for further applications.  A nilpotent Lie group $\mathcal{N}$  of step two is  a   vector space $\mathbb{R}^{{2n}}\times\mathbb{R}^r$ with the group multiplication given by
  \begin{equation} \label{eq:multipler}
  (\mathbf{x},\mathbf{t})\cdot (\mathbf{y} ,\mathbf{s} )=\left(\mathbf{x}+\mathbf{y} ,\mathbf{t} +\mathbf{s}  + 2  B(\mathbf{x},\mathbf{y} )  \right),
  \end{equation}
  where $B: \mathbb{R}^{2n}\times \mathbb{R}^{2n}\rightarrow \mathbb{R}^r$ is a skew-symmetric mapping. The sub-Laplacian on  $\mathcal{N}$  is the differential operator
  \begin{equation}\label{eq:L}
\Delta_b:=-\frac 14\sum_{k=1}^{{2n}} \mathbf{Y}_{k} \mathbf{Y}_{ k},
 \end{equation}
 where $\mathbf{Y}_k$'s  are standard left vector fields on $\mathcal{N}$ (cf. \eqref{eq:Y}). Let
 \begin{equation*}\label{eq:T}
\mathbf{ T}=(\mathbf{T}_1,\ldots,\mathbf{T}_r):=\left(\frac{\partial}{\partial t_1},\ldots,\frac{\partial}{\partial t_r}\right).
 \end{equation*}
 The spectrum of $\Delta_b$ alone is the set of nonnegative numbers
  $\{\lambda \in \Bbb R: \lambda \geq 0\}$ while the spectrum of $i \mathbf{T}$ is the set  $\Bbb R^r$. These operators are commutative, then their joint spectrum of $(\Delta_b, i \mathbf{T})$ is well defined and contains
 \begin{equation}\label{brush}
 \sigma(\Delta_b, i \mathbf{T})=\bigcup_{\mathbf{k}=(k_1,\ldots,k_n) \in\left(\mathbb{Z}_{+}\right)^{n}}\left\{(\lambda, -\tau) \in \Bbb R\times\Bbb R^{r}: \lambda=\sum_{j=1}^n   \mu_j(\tau)  (2k_j+1)\right\},
 \end{equation}
 where $\mu_j(\tau)$'s are defined by Proposition \ref{prop:orthonormal-basis}.

Laguerre calculus on non-degenerate nilpotent Lie group of step two was developed in  \cite{CMW}. It is more complicated compared to the Heisenberg group case, since we need to use $\tau$-coordinates for each $\tau$ in the second step $\Bbb R^r$ to define Laguerre distribution.  Let $\widetilde{\mathscr L}_{\mathbf{k} }^{(\mathbf{0} )}(\mathbf{y}, \tau)$ be the partial Fourier transform of Laguerre distribution on $\mathcal{N}$. Then the function
  \begin{equation}\label{joint-eigen-psi}
 \Psi_{\mathbf{k}, \tau}^{(\lambda)}(\mathbf{y}, \mathbf{t}):=e^{i\mathbf{t}\cdot \tau}\widetilde{\mathscr L}_{\mathbf{k} }^{(\mathbf{0} )}(\mathbf{y}, \tau)
 \end{equation}
 is smooth in $(\mathbf{y}, \mathbf{t})\in \Bbb R^{2n+r}$ for fixed $\tau \in \Bbb R^r$ and is the eigenfunction corresponding to the joint spectrum $(\lambda, -\tau) \in \sigma(\Delta_b, i \mathbf{T})$, that is,
\begin{equation*}\label{joint-eigen}
\begin{aligned}
 &(\lambda \mathbf{I}-\Delta_b) \Psi_{\mathbf{k}, \tau}^{(\lambda)}=\mathbf{0},\\
 & (-\tau_{\beta} \mathbf{I}-i \mathbf{T}_{\beta}) \Psi_{\mathbf{k}, \tau}^{(\lambda)}=\mathbf{0},\,\,\,\,\beta=1,\ldots,r,
 \end{aligned}
 \end{equation*}
 (cf. Proposition \ref{pro:joint-eigen}).
For fixed $m$, the spectral projection operator on $\mathcal{N}$ is defined formally as
 \begin{equation*}\label{spe-op.}
 (\mathbb{P}_{m}\phi)(\mathbf{y}, \mathbf{t}):=\int_{\Bbb R^{r}} \phi\ast\sum_{|\mathbf{k}|=m}\Psi_{\mathbf{k}, \tau}^{(\lambda)}(\mathbf{y}, \mathbf{t})d\tau,
\end{equation*}
where $(\lambda, -\tau) \in \sigma(\Delta_b, i \mathbf{T})$.

To define the operator $\mathbb{P}_{m}$ explicitly, we introduce a kernel in terms of  Laguerre functions
\begin{equation}\label{Qm}
 \begin{aligned}
 	Q_{m}(\mathbf{y}, \tau) :=\sum_{|\mathbf{k}|=m}\widetilde{\mathscr L}_{\mathbf{k} }^{(\mathbf{0} )}(\mathbf{y}, \tau),
 \end{aligned}
 \end{equation}
 which is proved to be integrable in $\tau$ for $\mathbf{y}\neq \mathbf{0}$.  So we can define its inverse   partial Fourier transform as
 \begin{equation}\label{Pm-Qm}
  P_{m}(\mathbf{y}, \mathbf{t}):=\frac{1}{(2 \pi)^{r}} \int_{\Bbb R^r}e^{i \mathbf{t}\cdot \tau} Q_{m}(\mathbf{y}, \tau) d \tau,\,\,\,\,\text{for}\,\,\mathbf{y}\neq \mathbf{0}.
  \end{equation}

On the other hand, we introduce a distribution $\mathscr{P}_m$ on $\mathcal N$ defined by the pair
\begin{equation}\label{def-dis-pm}
\left\langle\mathscr{P}_m,\phi\right\rangle:=\frac{1}{(2\pi)^r}\int_{\Bbb R^{2n+r}}Q_m(\mathbf{y},\tau)\widetilde{\phi}(\mathbf{y},-\tau) d\mathbf{y}d\tau,
\end{equation}
for $\phi\in \mathcal{S}(\mathcal{N})$, the space of Schwartz functions on $\mathcal{N}$, where
$\widetilde{\phi}$ is the partial Fourier transformation of  function $\phi$.
It is proved to be a homogeneous distribution of degree $-Q$, where $Q:=2n+2r$ (cf. Proposition \ref{pm-homo}).
Then the {\it  spectral projection operator} is defined as
\begin{equation*}
(\mathbb{P}_m\phi)(\mathbf{y},\mathbf{t}):=\phi*\mathscr{P}_m (\mathbf{y},\mathbf{t}),
\end{equation*}
for $\phi\in \mathcal{S}(\mathcal{N})$. We can give the integral representation of $P_m(\mathbf{y}, \mathbf{t})$ for $\mathbf{y} \neq \mathbf{0}$.
\begin{pro}\label{pro:pm-expression}
On a non-degenerate nilpotent Lie group of step two $\mathcal{N}$,  $P_m(\mathbf{y}, \mathbf{t})$  has the integral representation:
 \begin{equation}\label{eq:ex-pm}
 \begin{aligned}
  P_{m}(\mathbf{y}, \mathbf{t})
     = &\sum_{j=0}^r C_{m,j}\mathcal{I}_{m,j}(\mathbf{y}, \mathbf{t}), \,\,\text{for}\,\, \mathbf{y}\neq \mathbf{0},
  \end{aligned}\end{equation}
  where
    \begin{equation}\label{eq:ex-I-1}\begin{aligned}
 \mathcal{I}_{m,j}(\mathbf{y}, \mathbf{t})=\int_{S^{r-1}}
  \frac{(\det\mathcal{B}^\tau)^{\frac{1}{2}}\left(\langle\mathcal{B}^\tau\mathbf{y},\mathbf{y}\rangle+i\mathbf{ t} \cdot \tau\right)^{m-j}}
  {\left(\langle\mathcal{B}^\tau\mathbf{y},\mathbf{y}\rangle-i \mathbf{t }\cdot \tau\right)^{m+n+r-j}} d\tau,
  \end{aligned}\end{equation}
 and    $$\mathcal{B}^\tau:=[(B^\tau)^tB^\tau]^{\frac{1}{2}}$$ is a $2n\times 2n$ symmetric matrix, and constants $C_{m,j}$ given by \eqref{mj} depend on $m,j,n,r$.
\end{pro}
This is a very simple formula since Laguerre distributions are defined in terms of $\tau$-coordinates $\mathbf{y}^{\tau}$, while \eqref{eq:ex-pm}-\eqref{eq:ex-I-1} only depend on $\mathbf{y}$, not $\mathbf{y}^{\tau}$.
 $P_m$  can be further continuated to a Lipschitzian function on $\mathcal{N} \backslash \{\mathbf{0}\}$ by analytic continuation (cf. Theorem \ref{th:pm-ex}).
\begin{thm}\label{pf=f+pvf}
Let $\mathcal{N}$ be a non-degenerate nilpotent Lie group of step two.
 The distribution $\mathscr{P}_{m}$ and $p.v. P_{m}$ on  $\mathcal{N}$ satisfy
\begin{equation}\label{eq:pf=f+pvf}
\mathscr{P}_{m}= p.v.P_{m} +C\delta_{\mathbf{0}},
\end{equation}
for some constant $C$, where $\delta_{\mathbf{0}}$ is the delta distribution at $\mathbf{0}$ and for $\phi\in \mathcal{S}(\mathcal{N})$,
\begin{equation}\label{defi-p.v.pm}
\langle  p.v. P_{m}, \phi\rangle=\lim_{\varepsilon\rightarrow 0^+}\int_{\Bbb R^{2n+r}\backslash B(\mathbf{0},\varepsilon)}
\phi(\mathbf{y},\mathbf{t})P_m(\mathbf{y},\mathbf{t})d\mathbf{y}d\mathbf{t},
\end{equation}
where $B(\mathbf{0},\varepsilon):=\{(\mathbf{y},\mathbf{t})\in \mathcal{N}: (|\mathbf{y}|^4+|\mathbf{t}|^2)^{\frac{1}{4}}\leq \varepsilon\}$. $P_{m}$ has mean value zero, that is,
\begin{equation}\label{mu-k=0}
  \int_{\partial B(\mathbf{0},1)}P_m(\mathbf{y},\mathbf{t})d\mathbf{y}d\mathbf{t}=0.
\end{equation}
\end{thm}

 \begin{thm}\label{thm:con-Pm-bounded}
 Let $\mathcal{N}$ be a non-degenerate nilpotent Lie group of step two. Then $\mathbb{P}_{m}$ is a Calder\'{o}n-Zygmund  operator and then can be extended to a bounded operator from the $L^p(\mathcal{N})$ space to itself. Moreover, they are  mutually orthogonal project operators in $L^2(\mathcal{N})$, that is,
 \begin{equation*}
   \mathbb{P}_{m_1}\circ \mathbb{P}_{m_2}=\delta_{m_1}^{m_2}\mathbb{P}_{m_1}.
 \end{equation*}
\end{thm}
Chang-Markina-Wang \cite[Theorem 4.1]{CMW} showed that  for $\phi \in \mathcal{S}(\mathcal{N})$,
\begin{equation*}
\lim _{R \rightarrow 1^-} \sum_{|\mathbf{k}|=0}^{\infty} \phi \ast \mathscr L_{\mathbf{k} }^{(\mathbf{0} )}R^{|\mathbf{k}|}=\phi.
\end{equation*}
We  generalize this convergence theorem to $L^p(\mathcal{N})$ functions.
\begin{thm}\label{thm:f-decom}
 Let $\mathcal{N}$ be a non-degenerate nilpotent Lie group of step two. Then for $\phi \in L^{p}(\mathcal{N}), 1<p<\infty$,
  we have  $\sum_{m=0}^{\infty}R^{m}\mathbb{P}_{m}\phi\in L^p(\mathcal{N})$ and
\begin{equation*}\label{Pf--f}
\lim _{R \rightarrow 1^-} \sum_{m=0}^{\infty} R^{m}\mathbb{P}_{m}\phi=\phi,
\end{equation*}
exists in $L^{p}(\mathcal{N})$.
\end{thm}
On $L^2(\mathcal{N})$, we have the following decomposition of functions without using Abel sum.
\begin{cor}\label{cor:f-decom}
Let $\mathcal{N}$ be a non-degenerate nilpotent Lie group of step two. Then for  $\phi\in L^2(\mathcal{N})$,   we have
 \begin{equation}\label{L2-decom}
\sum_{m=0}^{\infty} \mathbb{P}_{m} \phi =\phi,
\end{equation}
in $L^2(\mathcal{N})$.
\end{cor}
As we almost accomplish the proofs of main theorems, we noted the reference \cite{narayanan}, where the authors studied the spectral projection operators on  H-type groups, a special kind of nilpotent Lie groups of step two. For this kind of groups, eigenfunctions of joint spectrum can be written down directly as
\begin{equation*}\begin{aligned}
\Delta_b u=(2k+n)|\tau|u,
\end{aligned}\end{equation*}
with
$$u(\mathbf{y},\mathbf{t}):=e^{-i\tau\cdot \mathbf{t}}L_k^{n-1}\left(\frac{|\tau||\mathbf{y}|^2}{2}\right)e^{-\frac{1}{4}|\tau||\mathbf{y}|^2}$$
 by using properties of  the Laguerre functions.
That means $u$ is an eigenfunction of the sub-Laplacian operator $\Delta_b$ with the eigenvalue $2k+n$ and by using explicit integrals, they can also write down the kernel of spectrum projection operator in a very concrete way.  But for general non-degenerate nilpotent Lie group of step two, we have
\begin{equation*}\begin{aligned}
\Delta_b \left(e^{i\mathbf{t}\cdot \tau}\widetilde{\mathscr L}_{\mathbf{k} }^{(\mathbf{0} )}(\mathbf{y}, \tau)\right)=\sum_{j=1}^n \mu_j(\tau)(2k_j+1)e^{i\mathbf{t}\cdot \tau}\widetilde{\mathscr L}_{\mathbf{k} }^{(\mathbf{0} )}(\mathbf{y}, \tau).
\end{aligned}\end{equation*}
 It is worth mentioning that eigenvalues depend on $\tau$, while for H-type groups, we always have $\mu_j(\tau)\equiv1$. We have to  construct eigenfunctions by using Laguerre calculus, which is more complicated in the general step two case.  The difficulty comes from the fact that the Laguerre functions $\mathscr L_{\mathbf{k} }^{(\mathbf{0} )}(\mathbf{y}, \tau)$'s are only distrubutions and their definition depend on $\tau$-coordinates, i.e. depend on directions in $\Bbb R^r$. Their behavior is so bad that it is difficult to construct the kernel $P_m$ of the spectral projection operator $\mathbb{P}_m$ and prove it to be Lipschitzian on $\mathcal{N}\backslash \{\mathbf{0}\}$. But for $\mathbf{y}\neq 0$, we still have a simple integral representation formula for $P_m$, and have to use analytic continuation for $\mathbf{y}=0$.  For general nondegenerate nilpotent Lie groups of step two, the substantial work to do is to construct the spectral projection operators and their kernels.

The paper is organized as follows. In Section \ref{laguerre-dis}, we recall the Laguerre calculus on non-degenerate nilpotent Lie group of step two. In Section \ref{section6}, we give the joint spectrum of $(\Delta_b, i \mathbf{T})$ and  eigenfunctions $\Psi_{\mathbf{k},\tau}^{(\lambda)}(\mathbf{y}, \mathbf{t})$ corresponding to $(\lambda, -\tau) \in \sigma(\Delta_b, i \mathbf{T})$ and  the relation between function $\Psi_{\mathbf{k},\tau}^{(\lambda)}(\mathbf{y}, \mathbf{t})$  and $P_m$. In Section \ref{section3}, we give the integral representation of $P_m$ for $\mathbf{y}\neq\mathbf{0}$ and  further continuate it  to a Lipschitzian function on $\mathcal{N} \backslash \{\mathbf{0}\}$ by analytic continuation. In Section \ref{section4}, we prove that the distribution $\mathscr{P}_m$ is homogeneous of degree $-Q$ and  is a homogeneous extension of $P_m$, and $ \mathbb{P}_{m}$ is a spectral projection operator and a Calder\'{o}n-Zygmund operator. So it can be extended to a bounded operator from the $L^p(\mathcal{N})$ to itself. In Section \ref{section5},  another kind of principal value $\widehat{p.v.} P_m$ is introduced  in order to  estimating the boundedness of the operator associated with the Abel sum $\sum_{m=0}^{\infty}R^m\mathbb{P}_{m}$. We give the explicit bound of the operator  norm of $\widehat{p.v.} P_m$ on $L^{p}(\mathcal{N})$ and $ \widehat{p.v.}\sum_{m=0}^{\infty}R^mP_{m} $ on $L^{p}$ respectively. Then we give the proof of Theorem \ref{thm:f-decom}. For  $\phi \in L^2(\mathcal{N})$, we prove the decomposition $\sum_{m=0}^{\infty} \mathbb{P}_{m} \phi =\phi$  without Abel sum.

\section{preliminary}
\subsection{Laguerre distribution on non-degenerate nilpotent Lie group of step two}\label{laguerre-dis}

 Laguerre calculus is the symbolic tensor calculus originally  induced in terms of the Laguerre functions on the Heisenberg group $\mathcal{H}_n$. It was first introduced on $\mathcal{H}_1$ by Greiner~\cite{Greiner} and later extended to
$\mathcal{H}_n$ and $\mathcal{H}_n\times \mathbb{R}^m$ by Beals, Gaveau, Greiner and Vauthier~\cite{BGGV}\cite{BGV}.
The Laguerre functions have been used in the study of the twisted convolution, or equivalently, the Heisenberg convolution for several decades.
For example, Geller~\cite{Ge} found a formula that expressed the group Fourier transform of radial functions on the isotropic Heisenberg group in terms of Laguerre transform, while Peetre~\cite{Pe} derived the relation between the Weyl transform and Laguerre calculus.
 See Chang-Chang-Tie \cite{CCT}, Chang-Greiner-Tie \cite{CGT}, Tie \cite{T2} for the application to find the inversion of differential operators, and Chang-Tie \cite{CT} for the study of the associated spectral projection operators. The Laguerre calculus was developed for non-degenerate nilpotent Lie groups of step two in Chang-Markina-Wang \cite{CMW}. It was used in  \cite{CKW} to  find the explicit formulas for the heat kernel of sub-Laplacian operator and the fundamental solution of power of sub-Laplacian operator on this kind of  groups.

The skew-symmetric mapping $B: \mathbb{R}^{2n}\times \mathbb{R}^{2n}\rightarrow \mathbb{R}^r$ in \eqref{eq:multipler} can be written as
 \begin{equation*}\label{eq:Phi-B}
B(\mathbf{x},\mathbf{y} ):=\left(B^1(\mathbf{x},\mathbf{y} ),\ldots,B^r(\mathbf{x},\mathbf{y} )\right),\qquad B^\beta(\mathbf{x},\mathbf{y} ):=  \sum_{ k,l=1}^{2n}    B^\beta_{kl}  x_ky_l,
\end{equation*}
for $\mathbf{x}=(x_1,\ldots,x_{2n}), \mathbf{y}=(y_1,\ldots,y_{2n})\in \Bbb R^{2n}$.
For any $\tau=(\tau_1,\ldots, \tau_r)\in  \mathbb{R}^r \setminus\{\mathbf{0}\}$, denote a  skew-symmetric $2n\times 2n$ matrix
 \begin{equation*}\label{B-tau}
  B^\tau:= \sum_{\beta=1}^r\tau_\beta B^\beta.
 \end{equation*}
 For simplicity, we assume that $B^{\tau}$ is non-degenerate for all $\tau\in \Bbb R^r$ in this paper and the corresponding nilpotent Lie group of step two is called {\it non-degenerate}. Vector fields
\begin{equation}\label{eq:Y}
    \mathbf{Y}_k:=\partial_{y_k}+ 2  \sum_{\beta=1}^{r}\sum_{l=1}^{2n}  B^\beta_{lk} y_{l}
 \partial_{ t_\beta}, \qquad k=1,\ldots,2n,
 \end{equation}
are left invariant vector fields on $\mathcal{N}$.

 Chang-Markina-Wang \cite{CMW} showed that for $\tau\in S^{r-1}\setminus E$, where $E\subset S^{r-1}$ is of Hausdorff dimension at most $r-2$, there exists an orthonormal basis locally normalizing $B^{\tau}$.

 \begin{pro} \cite[Proposition 2.3]{CMW}\label{prop:orthonormal-basis}
There exists a subset $E \subseteq S^{r-1}$ of Hausdorff dimension     at most $   r-2 $, such that for any   $\tau_0\in  S^{r-1}\setminus E$, we can find a neighborhood $U$ of $\tau_0\in  S^{r-1}$ and an   orthonormal basis
$\{v^\tau_1,$ $\ldots, v^\tau_{{2n}}\}$ of $\mathbb{R}^{{2n}}$, smoothly depending on $\tau \in U$,  such that the matrix
$O(\tau)= (v^\tau_1,\ldots, v^\tau_{{2n}})$ normalizes $B^\tau$, i.e.
\begin{equation}\label{eq:J}
O(\tau)^t  B^\tau O(\tau)=J(\tau) :=\left(
  \begin{array}{ccccc}0& {-\mu_1(\tau)}&0&0&\cdots\\\mu_1(\tau)&0&0&0 &\cdots \\
 0&0&0&- \mu_2(\tau)&\cdots\\0&0&\mu_2(\tau)&0&\cdots\\\vdots&\vdots&\vdots&\vdots& \ddots
   \end{array}
\right),
\end{equation} where $\mu_1(\tau)\geq \mu_2(\tau)\geq\cdots\geq \mu_n(\tau)\geq 0$  also smoothly depend on $\tau$ in this neighborhood, and $i\mu_1(\tau),-i\mu_1(\tau),\ldots,i\mu_n(\tau),-i\mu_{ n}(\tau)$ represent repeated pure imaginary eigenvalues of $B^\tau$.
\end{pro}

 We  can write a point $\mathbf{y}\in \mathbb{R}^{2n}$ in terms of the basis $\{v^\tau_k\}$ as
 \begin{equation}\label{eq:y'-''}
   \mathbf{y}=\sum_{k=1}^{2n} y_k^\tau v^\tau_k    \in \mathbb{R}^{2n}
 \end{equation} for some $y_1^\tau,\ldots,y_{2n}^\tau\in \mathbb{R}$. We call $(y_1^\tau,\ldots,y_{2n}^\tau)$ the {\it $\tau$-coordinates} of the point $\mathbf{y}\in \mathbb{R}^{2n}$.

   Chang-Markina-Wang \cite{CMW}  defined Laguerre distributions by using $\tau$-coordinates.
Let $L_k^{(p)}$ be the generalized Laguerre polynomials.
 It is well known \cite{BCT} that
\begin{equation*}\label{eq:l-L}
l_k^{(p)}(\sigma):=\left [\frac {\Gamma(k+1)}{\Gamma(k+p+1)} \right]^{\frac 12}L_{k}^{(p)}( \sigma)\sigma^{\frac p2} e^{-\frac \sigma2},
\quad
\end{equation*}
 where $\sigma\in [0,\infty),  k, p\in \mathbb {Z}_{\geq 0} $,  constitute an orthonormal basis of $L^2([0,\infty),d\sigma)$ for fixed  $p$.
 Let  $ \mathscr  L_{k}^{(p)}$ on $\mathbb{R}^2 \times \mathbb{R}^r$ be the distribution defined via their partial Fourier transformations
\begin{equation}\label{eq:exponential-Laguerre0}
 \widetilde{\mathscr  L}_{k}^{(p)}(y,\tau ): =\frac {2 |\tau|}{\pi} ({\rm sgn}\, p)^p l_k^{(|p|)}(2|\tau| |y|^2)e^{ip\theta},
\end{equation}
which is smooth on $(\Bbb R^2 \times \Bbb R^r)\setminus\{\mathbf{0}\} $, where $p\in \mathbb{Z}$,  $ y=(y_1,y_2)\in \mathbb{R}^2, y=y_1+iy_2=|y|e^{i\theta},  \tau\in \mathbb{R}^r  $. Here the {\it
partial Fourier transformation} of a function $\phi$ is defined as
\begin{equation*} \label{eq:partial-Fourier}
   \widetilde{\phi}(\mathbf{y},\tau):=\int_{\mathbb{R}^r}e^{-i\tau \cdot \mathbf{t}}\phi(\mathbf{y},\mathbf{t})d\mathbf{t}, \qquad {\rm for}\quad \tau \in \mathbb{R}^r.
\end{equation*}
The \emph{partial Fourier transformation of a distribution} $\mathscr{K}$ is defined as
\begin{equation*}
\langle\widetilde{\mathscr{K}},\phi \rangle :=\langle \mathscr{K},  \widetilde{\phi}\rangle,
\end{equation*}
for $\phi \in \mathcal{S}(\mathcal{N})$.
 Define the {\it exponential Laguerre distributions} ${ \mathscr L}_{\mathbf{k} }^{(\mathbf{p} )}(y,s)$ on $\mathbb{R}^{2n+r}$ via their partial Fourier transformations
\begin{equation}\label{eq:exponential-Laguerre}
 \widetilde{ \mathscr L}_{\mathbf{k} }^{(\mathbf{p} )}(\mathbf{y},\tau ):=\prod_{j=1}^n \mu_j(\dot{\tau})\widetilde{\mathscr  L}_{k_j}^{(p_j)} \left(\sqrt {\mu_j(\dot{\tau})}{\mathbf y^\tau_j}, \tau\right),
\end{equation}
where $ \mathbf{y}\in \mathbb{R}^{2n},\tau\in \mathbb{R}^r, \mathbf{p}=(p_1,\ldots p_n)\in \mathbb{Z}^n  , $ $\mathbf{k}=(k_1,\ldots k_n)\in \mathbb{Z}^n_{\geq 0} $, and
\begin{equation*}\label{eq:mu-dot}
\dot{ \tau}=\frac \tau{|\tau|}\in S^{r-1},\qquad
   \mu_j(\tau)=|\tau|\mu_j(\dot{\tau}),
   \end{equation*}
   \begin{equation*}\label{eq:y-j}
{\mathbf y^{\tau}_j} = \left ( y_{2j-1}^\tau, y_{2 j}^\tau\right)\in \mathbb{R}^2,\qquad  {j}=1,\ldots,n.
\end{equation*}

The definition of $\widetilde{ \mathscr L}_{\mathbf{k} }^{(\mathbf{p} )}(\mathbf{y},\tau )$ above depends on the choice of local orthonormal basis normalizing $B^{\tau}$, and in that local neighborhood, it smoothly depends on $\mathbf{y}$ and $\tau$.  Then  $\widetilde{ \mathscr L}_{\mathbf{k} }^{(\mathbf{p} )}(\mathbf{y},\tau )$ is locally integrable and so $\mathscr L_{\mathbf{k} }^{(\mathbf{p} )}$ as a distribution is well defined.
On the other hand, for any fixed $\tau\in \Bbb R^r \backslash \{\mathbf{0}\}$ with $B^{\tau}$ non-degenerate, $\widetilde{ \mathscr L}_{\mathbf{k} }^{(\mathbf{p} )}(\cdot,\tau )$ is a Schwartz function over $\Bbb R^{2n}$ for fixed $\tau, \mathbf{k}$ and $\mathbf{p}$. $\{\widetilde{ \mathscr L}_{\mathbf{k} }^{(\mathbf{p} )}(\cdot,\tau )\}_{\mathbf{p}\in \Bbb{Z}^n},\mathbf{k}\in \Bbb{Z}^n_{\geq 0}$ constitute an orthogonal basis of $L^2(\Bbb R^{2n})$ and behave nicely under the {\it twisted convolution} of two functions  $f,g\in L^1(\mathbb{R}^{{2n}})$, which is defined as
\begin{equation}\label{twisor-con}
\begin{split}
    f*_\tau g( \mathbf{y})&=\int_{\mathbb{R}^{2n}} e^{-2iB^\tau(\mathbf{y},\mathbf{x})} f(\mathbf{y}-\mathbf{x}) g(\mathbf{x})d\mathbf{x},
\end{split}\end{equation}for any fixed $\tau\in \mathbb{R}^r $.
\begin{pro}\cite[Corollary 3.1]{CMW} \label{cor:Minkowski} For $1\leq p\leq\infty$, we have
   \begin{equation*}\label{eq:Minkowski}
      \|u*_\tau v\|_{L^p(\mathbb{R}^{{2n}})}\leq \|u \|_{L^1(\mathbb{R}^{{2n}})}\|  v\|_{L^p(\mathbb{R}^{{2n}})}.
   \end{equation*}
\end{pro}
\begin{pro}\cite[Proposition 1.2]{CMW}\label{Proposition 1.2 in 7}
For $\mathbf{k},  \mathbf{p},\mathbf{q}, \mathbf{m}\in \mathbb{Z}_+^n$, we have \begin{equation*}
\label{eq:product}
\widetilde{\mathscr  L}_{(\mathbf{k}\wedge \mathbf{p})-\mathbf{1} }^{(\mathbf{p} -\mathbf{k})}\ast_\tau\widetilde{\mathscr  L}_{(\mathbf{q}\wedge \mathbf{m})-\mathbf{1} }^{(\mathbf{q} -\mathbf{m})}
      =\delta_{\mathbf{k}}^{(\mathbf{q})}\widetilde{\mathscr  L}_{(\mathbf{p}\wedge \mathbf{m})-\mathbf{1} }^{(\mathbf{p} -\mathbf{m})},
         \end{equation*}
    where     $\mathbf{p}\wedge \mathbf{m}-\mathbf{1}:=( \min(k_1,p_1)-1,\dots, \min(k_n,p_n)-1) $ and $\delta_{ \mathbf{{k}}}^{( \mathbf{{q}})} $ is the Kronecker delta function.
         \end{pro}
\begin{lem}\label{lem:L1,L2}\cite[Lemma 4.1]{CMW} On a non-degenerate nilpotent Lie group of step two, for almost all  $\tau \in \mathbb{R}^r$, we have
$$
\begin{aligned}
& \left\|\widetilde{\mathscr{L}}_{\mathbf{k}}^{(\mathbf{p})}(\cdot, \tau)\right\|_{L^2\left(\mathbb{R}^{2 n}\right)}^2=\frac{2^n\left(\operatorname{det}\left|B^\tau\right|\right)^{\frac{1}{2}}}{\pi^n}=\frac{2^n}{\pi^n} \prod_{j=1}^n \mu_j(\tau), \\
& \left\|\widetilde{\mathscr{L}}_{\mathbf{k}}^{(\mathbf{p})}(\cdot, \tau)\right\|_{L^1\left(\mathbb{R}^{2 n}\right)}=\prod_{j=1}^n\left\|l_{k_j}^{\left(p_j\right)}\right\|_{L^1\left(\mathbb{R}^1\right)}.
&
\end{aligned}
$$
\end{lem}

\subsection{The  joint spectrum of $(\Delta_b, i \mathbf{T})$}\label{section6}

 For a differential operator $D$ on the group $\mathcal{N}$, we denote by $\widetilde{D}$ the \emph{partial symbol} of $D$ with respect to $\tau\in \Bbb R^r$, i.e., $\partial_{t_{\beta}}$ is replaced by $i\tau_{\beta}$. Then
\begin{equation}\begin{aligned}\label{eq:Y-tilde}
    \widetilde{\mathbf{Y}}_k&=\partial_{y_k}+ 2i  \sum_{\beta=1}^{r}\sum_{l=1}^{2n}  B^\beta_{lk} y_{l} \tau_{\beta}, \qquad k=1,\ldots,2n,\\
    i\widetilde{\mathbf{T}}_{\beta}&=-\tau_{\beta},\qquad \qquad\qquad\qquad\qquad\beta=1,\ldots,r.
 \end{aligned}\end{equation}

 Since $\Delta_b$ and $i\mathbf{T}$ are essentially self-adjoint strongly commuting operators, there is a well defined joint spectrum. Similarly to the joint spectral theory  on the Heisenberg group investigated by
Strichartz \cite{S}, the  joint spectrum  on nilpotent Lie group of step two is defined as
 the complement of the set of $(\lambda, \tau)\in \Bbb C^{r+1}$ where neither $\lambda I-\Delta_b$ or $\tau \mathbf{I}_\beta-i\mathbf{T}_\beta$ is invertible. The  joint spectrum of $(\Delta_b, i \mathbf{T})$ contains the union \eqref{brush} by the following proposition.

\begin{pro}\label{pro:joint-eigen}
 The function defined by \eqref{joint-eigen-psi} for $(\lambda, -\tau) \in \sigma(\Delta_b, i \mathbf{T})$
 is the eigenfunction of $\Delta_b$ and $i\mathbf{T}$.
\end{pro}
\begin{proof} Recall that for $v=(v_1,\ldots, v_{2n})\in \Bbb R^{2n}$,
\begin{equation*}Y_v:=\sum_{k=1}^{2n} v_kY_k=\partial_v+  2   B(y,v)\cdot
\partial_t,
\end{equation*}
is a left invariant vector field on $\mathcal{N} $, where
$ B(y,v )\cdot\partial_t:= B^1 (y,v )\partial_{t_1}+\cdots +B^r(y,v )\partial_{t_r}$
and  $\partial_v$ is the derivative    on $\mathbb{R}^{{2n}}$ along the direction $v$, i.e.
$ \partial_v =\sum_{k=1}^{{2n}}   v_k\partial_{y_k}$.   Then  by \cite[Proposition 6.1]{CMW},
 \begin{equation*}
 \Delta_b=-\frac{1}{4}\sum_{j=1}^{2n}Y_{v_j^\tau}Y_{v_j^\tau},
 \end{equation*}
 where $\{v_1^\tau,\ldots, v_{2n}^\tau\}$ is the local orthonormal basis of $\Bbb R^{2n}$.
 Then its partial symbol is
  \begin{equation*}\label{symbol-delta}
 \widetilde{\Delta}_b=-\frac{1}{4}\sum_{j=1}^{2n}\widetilde{Y}_{v_j^\tau}\widetilde{Y}_{v_j^\tau}.
 \end{equation*}
It follows from (6.1) in \cite[P.1875]{CMW} that
  \begin{equation}\label{tilda-delta}
 \widetilde{\Delta}_b\widetilde{ \mathscr L}_{\mathbf{k} }^{(\mathbf{0} )}(\mathbf{y},\tau)= \sum_{j=1}^{n}\mu_j(\tau)(2k_j+1)\widetilde{ \mathscr L}_{\mathbf{k} }^{(\mathbf{0} )}(\mathbf{y},\tau).
 \end{equation}
For $(\lambda, -\tau) \in \sigma(\Delta_b, i \mathbf{T})$, we see that
 \begin{equation*}
 (\lambda-\widetilde{\Delta}_b)  \widetilde{ \mathscr L}_{\mathbf{k} }^{(\mathbf{0} )}(\mathbf{y},\mathbf{\tau})
 =\bigg(\lambda \mathbf{I}-\sum_{j=1}^n\mu_j(\tau)(2k_j+1)\bigg)\widetilde{\mathscr L}_{\mathbf{k}}^{(\mathbf{0} )}(\mathbf{y},\mathbf{\tau})=0,
   \end{equation*}
and
   \begin{equation*}
   (-\tau_{\beta} \mathbf{I}-i\widetilde{\mathbf{T}}_{\beta}) \widetilde{ \mathscr L}_{\mathbf{k} }^{(\mathbf{0} )}(\mathbf{y},\mathbf{\tau})=(-\tau_{\beta}+\tau_{\beta})  \widetilde{ \mathscr L}_{\mathbf{k} }^{(\mathbf{0} )}(\mathbf{y},\mathbf{\tau})=0,
   \end{equation*}
by using \eqref{tilda-delta} and the second identity of \eqref{eq:Y-tilde}.
  On the other hand, since
  \begin{equation*}\begin{aligned}
  \mathbf{Y}_k\left(e^{i\mathbf{t}\cdot \tau}\widetilde{ \mathscr L}_{\mathbf{k} }^{(\mathbf{0} )}(\mathbf{y},\tau)\right)
  =&e^{i\mathbf{t}\cdot \tau}\bigg(\partial_{y_k}+2i  \sum_{\beta=1}^{r}\sum_{l=1}^{2n}  B^\beta_{lk} y_{l}
 \tau_\beta\bigg)\widetilde{ \mathscr L}_{\mathbf{k} }^{(\mathbf{0} )}(\mathbf{y},\tau)\\
 =&e^{i\mathbf{t}\cdot \tau} \widetilde{\mathbf{Y}}_k\widetilde{ \mathscr L}_{\mathbf{k} }^{(\mathbf{0} )}(\mathbf{y},\tau),
  \end{aligned}\end{equation*}
  and \begin{equation*}\begin{aligned}
  \mathbf{Y}_k\mathbf{Y}_k\left(e^{i\mathbf{t}\cdot \tau}\widetilde{ \mathscr L}_{\mathbf{k} }^{(\mathbf{0} )}(\mathbf{y},\tau)\right)=&\bigg(\partial_{y_k}+ 2  \sum_{\beta=1}^{r}\sum_{l=1}^{2n}  B^\beta_{lk} y_{l}
 \partial_{ t_\beta}\bigg)e^{i\mathbf{t}\cdot \tau} \widetilde{\mathbf{Y}}_k\widetilde{ \mathscr L}_{\mathbf{k} }^{(\mathbf{0} )}(\mathbf{y},\tau)\\
 =&e^{i\mathbf{t}\cdot \tau}\widetilde{\mathbf{Y}}_k \widetilde{\mathbf{Y}}_k\widetilde{ \mathscr L}_{\mathbf{k} }^{(\mathbf{0} )}(\mathbf{y},\tau),
  \end{aligned}\end{equation*}
  then
  \begin{equation*}
  \Delta_b \bigg(e^{i\mathbf{t}\cdot \tau}\widetilde{ \mathscr L}_{\mathbf{k} }^{(\mathbf{0} )}(\mathbf{y},\tau)\bigg)=e^{i\mathbf{t}\cdot \tau}\widetilde{\Delta}_b\widetilde{ \mathscr L}_{\mathbf{k} }^{(\mathbf{0} )}(\mathbf{y},\tau),
  \end{equation*}
by using \eqref{eq:L}. This yields that
 \begin{equation*}\label{eigen-tildephi-1}
 \begin{aligned}
(\lambda \mathbf{I}-\Delta_b) \Phi_{\mathbf{k},\dot{\tau}}^{(\lambda)}(\mathbf{y},\mathbf{t} )
=&e^{i\mathbf{t}\cdot \tau}(\lambda-\widetilde{\Delta}_b)\widetilde{ \mathscr L}_{\mathbf{k} }^{(\mathbf{0} )}(\mathbf{y},\tau)=0,
 \end{aligned}\end{equation*}
and
   \begin{equation*}\label{eigen-tildephi-2}
 \begin{aligned}
 (-\tau_\beta \mathbf{I}-i \mathbf{T}_\beta) \Psi_{\mathbf{k},\tau}^{(\lambda)}(\mathbf{y},\mathbf{t} )=(-\tau_\beta+\tau_\beta)e^{i\mathbf{t}\cdot \tau}\widetilde{ \mathscr L}_{\mathbf{k} }^{(\mathbf{0} )}(\mathbf{y},\tau)=0.
   \end{aligned}\end{equation*}
  The proposition is proved.
\end{proof}

 \begin{cor}\label{phi-pm}
  Let $\mathcal{N}$ be  a non-degenerate nilpotent Lie group of step two. For $P_{m}$ defined by \eqref{Pm-Qm} and  the eigenfunction $\Psi_{\mathbf{k}, \tau}^{(\lambda)}$ defined by \eqref{joint-eigen-psi} corresponding to the joint spectrum $(\lambda, -\tau) \in \sigma(\Delta_b, i \mathbf{T})$, we have
 \begin{equation*}\label{p=p}
  \begin{aligned}
\frac{1}{(2\pi)^r}\int_{\Bbb R^{r}}\sum_{|\mathbf{k}|=m} \Psi_{\mathbf{k},\tau}^{(\lambda)}(\mathbf{y}, \mathbf{t})d\tau=P_{m}(\mathbf{y}, \mathbf{t}),\qquad \text{for}\,\,\mathbf{y}\neq 0.
 \end{aligned}
 \end{equation*}
\end{cor}
\begin{proof}
Since
  \begin{equation*} \begin{aligned}\label{sum-phi}
 \sum_{|\mathbf{k}|=m} \Psi_{\mathbf{k}, \tau}^{(\lambda)}(\mathbf{y}, \mathbf{t}) =e^{i\mathbf{t}\cdot \tau}Q_m(\mathbf{y}, \tau),
  \end{aligned}\end{equation*}
 by using \eqref{joint-eigen-psi} and \eqref{Qm}, then
 \begin{equation*} \begin{aligned}
 \frac{1}{(2\pi)^r}\int_{\Bbb R^{r}}\sum_{|\mathbf{k}|=m} \Psi_{\mathbf{k}, \tau}^{(\lambda)}(\mathbf{y}, \mathbf{t})d\tau
=\frac{1}{(2\pi)^r}\int_{\Bbb R^{r}}e^{i\mathbf{t}\cdot \tau}Q_m(\mathbf{y}, \tau)d\tau
=P_{m}(\mathbf{y}, \mathbf{t}),\qquad \text{for}\,\,\mathbf{y}\neq 0.
 \end{aligned}\end{equation*}
 Here $Q_m(\mathbf{y}, \tau)$ is integrable in $\tau$ for $\mathbf{y}\neq 0$ by following estimate \eqref{Qm-integrable}.
\end{proof}

\section{The kernel of spectral projection operator $\mathbb{P}_m$ }\label{section3}
\subsection{The kernel $P_m(\mathbf{y},\mathbf{t})$ for $\mathbf{y}\neq 0$}
\begin{lem}\label{lem:B-y}
On a  non-degenerate nilpotent Lie group $\mathcal{N}$ of step two, there exists a constant $C>1$  independent of $\tau$ and  $\mathbf{y}$ such that
\begin{equation*}
C^{-1} |\mathbf{y}|^2 \leq \left|\langle\mathcal{B}^\tau\mathbf{y},\mathbf{y}\rangle\right| \leq C |\mathbf{y}|^2,\qquad\text{for}\,\,\tau\in S^{r-1}.
\end{equation*}
\end{lem}
\begin{proof}
The lemma holds for $\mathbf{y}=0$ obviously.
For $\mathbf{y}\neq 0$,
$$(B^{\tau})^{t}B^{\tau}=O(\tau)J(\tau)^{t}J(\tau) O(\tau)^{t},$$
where $J(\tau)$ is defined by \eqref{eq:J} and $O(\tau)$ is the orthogonal matrix given by Proposition \ref{prop:orthonormal-basis}.
 Note that $J(\tau)^{t}J(\tau) =\Lambda^{2}(\tau)$, where $$\Lambda(\tau):={\rm diag}\bigg(\mu_{1}(\tau),\mu_{1}(\tau),\mu_{2}(\tau),\mu_{2}(\tau),\ldots,\mu_{n}(\tau),\mu_{n}(\tau)\bigg).$$
  So
\begin{equation*}\label{B-tau-||}
\mathcal{B}^\tau=[O(\tau)J(\tau)^{t}J(\tau) O(\tau)^{t}]^{\frac{1}{2}}=O(\tau)\Lambda(\tau)O(\tau)^{t}.
\end{equation*}
Namely, $\mu_j(\tau)$'s are eigenvalues of $\mathcal{B}^{\tau}$. Then
\begin{equation}\label{Byy}
\langle\mathcal{B}^\tau\mathbf{y},\mathbf{y}\rangle =\sum_{j=1}^n\mu_{j}(\tau)|\mathbf{y}_j^{\tau}|^{2}=|\tau|\sum_{j=1}^n\mu_{j}(\dot{\tau})|\mathbf{y}_j^{\dot{\tau}}|^{2},
\end{equation}
by Proposition \ref{prop:orthonormal-basis} and \eqref{eq:y'-''}.
Since  $B^\tau$ is linear in $\tau$ and  eigenvalues $\mu_j(\tau)$'s  depend continuously on $\tau$ and $\mu_j(\tau) >0$ by Proposition \ref{prop:orthonormal-basis} again,  then there exists a constant $C>1$ independent on $\tau$, such that
$$C^{-1}\leq \mu_{j}(\tau) \leq C,\qquad \text{for}\,\, \tau\in S^{r-1}.$$
Therefore $$ C^{-1}|\mathbf{y}|^2\leq \left|\langle\mathcal{B}^\tau\mathbf{y},\mathbf{y}\rangle\right|
=\left|\sum_{j=1}^n\mu_{j}(\tau)|\mathbf{y}_j^{\tau}|^2\right|\leq C |\mathbf{y}|^2.$$
The lemma is proved.
\end{proof}

\emph{Proof of Proposition \ref{pro:pm-expression}.}
The integral in \eqref{eq:ex-I-1} is well-defined for any $\mathbf{y}\neq \mathbf{0}$  since the denominator of the integrand does not vanish by  Lemma \ref{lem:B-y}.
Noting that \begin{equation*}
\begin{aligned}
 (\det\mathcal{B}^\tau)^{\frac{1}{2}}&=\prod_{j=1}^{n} \mu_{j}(\tau)=|\tau|^{n}(\det\mathcal{B^{\dot{\tau}}})^{\frac{1}{2}},
 \end{aligned}
   \end{equation*}
   by  Lemma 4.1 in \cite{CMW}, we have
    \begin{equation}\label{Qm-2}
 \begin{aligned}
 	Q_{m}(\mathbf{y}, \tau) &= \sum_{|\mathbf{k}|=m} \prod_{j=1}^{n} \mu_{j}(\dot{\tau}) \widetilde{\mathscr L}_{k_{j}}^{(0)}\left(\sqrt{\mu_{j}(\dot{\tau})} \mathbf{y}_j^{\tau}, \tau\right) \\
 	&=\frac{2^{n}(\det\mathcal{B}^\tau)^{\frac{1}{2}}}{\pi^{n}} e^{-\sum_{j=1}^n\mu_{j}(\tau)|\mathbf{y}_j^{\tau}|^{2}}\sum_{|\mathbf{k}|=m }\prod_{j=1}^{n}L_{k_{j}}^{(0)}\left(2 \mu_{j}(\tau)\left|\mathbf{y}_j^{\tau}\right|^{2}\right)  \\
 	&=\frac{2^{n}(\det\mathcal{B}^\tau)^{\frac{1}{2}}}{\pi^{n}} e^{-\langle\mathcal{B}^\tau\mathbf{y},\mathbf{y}\rangle } L_{m}^{(n-1)}(2\langle\mathcal{B}^\tau\mathbf{y},\mathbf{y}\rangle ),
 \end{aligned}
 \end{equation}
by using \eqref{Qm}, \eqref{eq:exponential-Laguerre0}, \eqref{eq:exponential-Laguerre} and  \eqref{Byy}. Here we use
 \begin{equation*}
 \sum_{|\mathbf{k}|=m }\prod_{j=1}^{n}L_{k_{j}}^{(0)}\left(\sigma_j\right)
 =L_{m}^{(n-1)}\bigg(\sum_{j=1}^n\sigma_{j}\bigg),\qquad \text{for}\,\, \sigma_j\in \Bbb R,
  \end{equation*}
  in the last identity, which can be followed by the fact that
  \begin{equation*}
 \sum_{j=1 }^nL_{n}^{(\alpha)}(x)L_{n-m}^{(\beta)}(y)
 =L_{n}^{(\alpha+\beta+1)}(x+y),
  \end{equation*}
  see \cite[P. 249]{special}.  Noting that  $\prod_{j=1}^{n} \mu_{j}(\dot{\tau})$ is bounded on $S^{r-1}$, there exists a constant $C>0$, such that
\begin{equation}\label{Qm-integrable}
\begin{aligned}
&\left|Q_m(\mathbf{y},\tau)\right|
\lesssim e^{-C^{-1}|\tau||\mathbf{y}|^{2}}|\tau|^{n}(C|\mathbf{\tau}||\mathbf{y}|^{2}+1)^m,
\end{aligned}
\end{equation}
by Lemma \ref{lem:B-y}, since $L_{m}^{(n-1)}(s)$ is a polynomial of degree $m$ in $s$. It follows that  $Q_{m}(\mathbf{y}, \cdot)$  for $\mathbf{y}\neq 0$ is integrable on $\Bbb R^r$ and we can take the partial inverse Fourier transform with respect to $\tau$ to get
\begin{equation}\label{eq:pm-r}
\begin{aligned}
 P_{m}(\mathbf{y}, \mathbf{t})&=\frac{1}{(2 \pi)^{r}} \int_{\Bbb R^r}e^{i \mathbf{t} \cdot \tau} Q_{m}(\mathbf{y}, \tau) d \tau\\
 &=\frac{1}{(2 \pi)^{r}} \int_{\Bbb R^r}e^{i \mathbf{t} \cdot \tau}\frac{2^{n}}{\pi^{n}}(\det\mathcal{B}^\tau)^{\frac{1}{2}}
 e^{-\langle\mathcal{B}^\tau\mathbf{y},\mathbf{y}\rangle } L_{m}^{(n-1)}(2\langle\mathcal{B}^\tau\mathbf{y},\mathbf{y}\rangle )d\tau\\
 &=\frac{ 2^{n-r}}{\pi^{n+r}}\int_{S^{r-1}} (\det\mathcal{B^{\dot{\tau}}})^{\frac{1}{2}}d\dot{\tau}\int_{0}^{\infty}|\tau|^{n+r-1}  e^{-|\tau|\sigma+i|\tau|\mathbf{t}\cdot \dot{\tau}} L_{m}^{(n-1)}(2|\tau|\sigma) d|\tau|\\
  &=\frac{1}{4^r\pi^{n+r}}\int_{S^{r-1}} \frac{(\det\mathcal{B^{\dot{\tau}}})^{\frac{1}{2}}}{\sigma^{n+r}}d\dot{\tau}\int_{0}^{\infty} s^{n+r-1} e^{-\frac{s}{2}\left(1-\frac{i \mathbf{t} \cdot \dot{\tau}}{\sigma}\right)} L_{m}^{(n-1)}(s) d s,
 \end{aligned}
 \end{equation}
  where we denote $\sigma:=\langle \mathcal{B^{\dot{\tau}}}\mathbf{y},\mathbf{y}\rangle>0$,
 and take coordinates transformation $s=2|\tau|\sigma$ in the last identity.
 Note that
 $$ \int_{0}^{\infty} e^{-w s} s^{k} L_{m}^{(k)}(s) d s=\frac{(m+k) !}{m !} \cdot \frac{(w-1)^{m}}{w^{m+k+1}}, \quad\text{for}\,\, {\rm Re} w\in \Bbb R_{+},$$
 see \cite[P. 244]{special} and
 $$\bigg(\frac{(m+n-1)!}{m!}\cdot\frac{(w-1)^m}{w^{m+n}}\bigg)^{(r)}=
 \sum_{j=0}^r(-1)^{r-j}\dbinom{r}{j}\frac{(m+n+r-j-1)!}{(m-j)!}\frac{(w-1)^{m-j}}{w^{m+n+r-j}}.$$
 Then we have
\begin{equation*}
 \begin{aligned}
 	\int_{0}^{\infty} s^{n+r-1} e^{-w s} L_{m}^{(n-1)}(s) d s
  &=(-1)^r\frac{d^r}{d w^r} \int_{0}^{\infty} s^{n-1} e^{-w s} L_{m}^{(n-1)}(s) d s \\
 	&=(-1)^r\frac{d^r}{d w^r}\left[\frac{(m+n-1) !}{m !} \cdot \frac{\left(w-1\right)^{m}}{w^{m+n}}\right] \\
 	&=\sum_{j=0}^r(-1)^{m}\dbinom{r}{j}\frac{(m+n+r-j-1)!}{(m-j)!}\cdot \frac{(1-w)^{m-j}}{w^{m+n+r-j}}.
 \end{aligned}
\end{equation*}
 Apply this identity to \eqref{eq:pm-r} with $w=\frac{1}{2}\left(1-\frac{i \mathbf{t} \cdot \dot{\tau}}{\sigma}\right)$ to get
 \begin{equation*}
\begin{aligned}
 P_{m}(\mathbf{y}, \mathbf{t})
  &=\frac{1}{4^r\pi^{n+r}}\sum_{j=0}^r(-1)^{m}\dbinom{r}{j}\frac{(m+n+r-j-1)!}{(m-j)!}\int_{S^{r-1}} \frac{(\det\mathcal{B^{\dot{\tau}}})^{\frac{1}{2}}}{\sigma^{n+r}} \frac{\left(\frac{\sigma+i \mathbf{t} \cdot \dot{\tau}}{2\sigma}\right)^{m-j}}{\left(\frac{\sigma-i \mathbf{t} \cdot \dot{\tau}}{2\sigma}\right)^{m+n+r-j}}d\dot{\tau}\\
  &=\sum_{j=0}^rC_{m,j}\int_{S^{r-1}} \frac{(\det\mathcal{B^{\dot{\tau}}})^{\frac{1}{2}}\left(\sigma+i \mathbf{t} \cdot \dot{\tau}\right)^{m-j}}{\left(\sigma-i \mathbf{t} \cdot \dot{\tau}\right)^{m+n+r-j}}d\dot{\tau},
 \end{aligned}
 \end{equation*}
 where $C_{m,j}$ is given by
   \begin{equation} \label{mj}
  C_{m,j}:=\frac{(-1)^{m}2^{n-r}}{\pi^{n+r}}\dbinom{r}{j}\frac{(m+n+r-1-j)!}{(m-j)!}.
  \end{equation}
Hence, \eqref{eq:ex-pm}-\eqref{eq:ex-I-1} hold.  The proposition is proved. \qed
  \subsection{The extension of $P_m$ to a Lipschitzian function on $\mathcal{N}\backslash \{\mathbf{0}\}$}
\begin{pro}\label{pro:B-tau-hol}
On a  non-degenerate nilpotent Lie group $\mathcal{N}$ of step two, $\mathcal{B}^{\tau}$ and $(\mathcal{B}^{\tau})^{\frac{1}{2}}$ are holomorphic in $\tau$ in a neighborhood $U$ of $S^{r-1} \subset \Bbb C^{r}$.
\end{pro}
\begin{proof}
 Suppose that $\phi$ is holomorphic in a domain $\Delta \subseteq \Bbb C$  containing all the eigenvalues of a $2n \times 2n$ matrix $T$, and let $\gamma \subset \Delta$  be a simple closed smooth curve with positive direction enclosing all the eigenvalues of $T$. Then
 \begin{equation*}\label{T-int}
 \phi(T)=\frac{1}{2\pi i}\int_{\gamma}\phi(z)(z I_{2n}-T)^{-1}dz,
 \end{equation*}
 (cf. \cite[P. 44]{kato}). Recall that for a matrix-valued function $T(x)$ continuous  in $x$,
 the eigenvalues of $T(x)$ is continuous in $x$ (cf.  \cite[P. 107] {kato}). $(B^{\tau})^{t}B^{\tau}$ is a quadratic function and so analytic in $\tau\in \Bbb C^r$.
 For $\tau\in S^{r-1}$, its eigenvalues are all positive by non-degeneracy. We may assume that for all $\tau\in S^{r-1}$,  $\mu_j(\tau)\in [a,b]$, for some  $a,b>0$.  This together with the continuity of eigenvalues implies that there exists a neighborhood $U$ of
 $S^{r-1}$, so that all eigenvalues of $(B^{\tau})^{t}B^{\tau}$ is contained in a curve
 $\gamma$ containing $[a,b]$ which lies in the right half space. Now taking $\phi(z)$ to be a branch of multi-valued holomorphic function $z^{\frac{1}{2}}$ on $\Bbb C \backslash (-\infty, 0)$ with $\phi(x)=x^{\frac{1}{2}}$ for $x\in (0, +\infty)$, we get
 \begin{equation*}
 \mathcal{B}^{\tau}=\frac{1}{2\pi i}\int_{\gamma}z^{\frac{1}{2}}(z I_{2n}-(B^{\tau})^{t}B^{\tau})^{-1}dz,
 \end{equation*}
  and then
 \begin{equation*}
 \frac{\partial}{\partial \overline{\tau}} \mathcal{B}^{\tau}=\frac{1}{2\pi i}\int_{\gamma}\frac{\partial}{\partial \overline{\tau}}z^{\frac{1}{2}}(z I_{2n}-(B^{\tau})^{t}B^{\tau})^{-1}dz=0.
 \end{equation*}
 Then $ \mathcal{B}^{\tau}$ is holomorphic in $\tau\in U$ because $(B^{\tau})^{t}B^{\tau}$ is holomorphic in $\tau\in U$.
Similarly, $(\mathcal{B}^{\tau})^{\frac{1}{2}}$ is also holomorphic in $\tau$ in a neighborhood $U$ of $S^{r-1} \subset \Bbb C^{r}$. The proposition is proved. \end{proof}

We can extend $P_m$ continuously to the subspace $\mathbf{y}=0$ except $(\mathbf{0},\mathbf{0})$ and get the explicit integral representation formula of  $P_m$ for $(\mathbf{y},\mathbf{t})\neq(\mathbf{0},\mathbf{0})$.
\begin{thm}\label{th:pm-ex}
   Let $\mathcal{N}$ be a non-degenerate nilpotent Lie group of step two. For any $\mathbf{t}_0 \in \Bbb R^r$, there exists a neighborhood  $U$ of $\mathbf{t}_0$ and  an orthogonal matrix $O_\mathbf{t}$, smoothly depending on $\mathbf{t}\in U$ and fixing $\mathbf{t}$, such that $P_m(\mathbf{y}, \mathbf{t})$ given by \eqref{eq:ex-pm} also has the following expression
  \begin{equation}\label{eq:ex-pm-2}
 \begin{aligned}
  P_{m}(\mathbf{y}, \mathbf{t})
     = &\sum_{j=0}^r C_{m,j}\widetilde{\mathcal{I}}_{m,j}(\mathbf{y}, \mathbf{t}),
  \end{aligned}\end{equation}
  where $C_{m,j}$ is given by \eqref{mj},
  \begin{equation}\label{eq:ex-I-2}
 \begin{aligned}
  \widetilde{\mathcal{I}}_{m,j}(\mathbf{y}, \mathbf{t})
  :=\int_{L_\varepsilon}dz\int_{S^{r-2}}
  \frac{(\det\mathcal{B}^{O_{\mathbf{t}}\chi_{\tau^{\prime}}(z)})^{\frac{1}{2}}
  \left(\langle\mathcal{B}^{O_{\mathbf{t}}\chi_{\tau^{\prime}}(z)}
  \mathbf{y},\mathbf{y}\rangle+i|\mathbf{t}|z\right)^{m-j}(1-z^2)^{\frac{r-2}{2}}d\tau^{\prime}}
  {\left(\langle\mathcal{B}^{O_{\mathbf{t}}\chi_{\tau^{\prime}}(z)}\mathbf{y},\mathbf{y}\rangle-i |\mathbf{ t}|z\right)^{m+n+r-j}},
  \end{aligned}\end{equation}
and  $L_{\varepsilon}$ is the changed contour  of $(-1,1)$ given by
\begin{equation*}\begin{aligned}
    L_{\varepsilon}:=\left\{z=x+iy\in \Bbb C, x\in (-1,1), y= \varepsilon \sqrt{1-x^2}\right\},\quad\text{for sufficiently  small}\,\,\varepsilon.
    \end{aligned}\end{equation*}
 $\chi_{\tau^{\prime}}$ is the mapping given by \eqref{chi}. $ \widetilde{\mathcal{I}}_{m,j}(\mathbf{y}, \mathbf{t})$ still  makes sense for $\mathbf{y}=0$ but $\mathbf{t}\neq 0$.
  \end{thm}
\begin{proof} Let $\varepsilon$ be a small positive constant which we will choose later.
 For the case $\mathbf{t}=|\mathbf{t}|(1,0,\ldots,0)$, it follows from \eqref{eq:ex-I-1} that
 \begin{equation*}
 \begin{aligned}
  \mathcal{I}_{m,j}(\mathbf{y}, \mathbf{t})&=\int_{S^{r-1}}
  \frac{(\det\mathcal{B}^\tau)^{\frac{1}{2}}\left(\langle\mathcal{B}^\tau\mathbf{y},\mathbf{y}\rangle+i|\mathbf{t}|\tau_1\right)^{m-j}}
  {\left(\left\langle\mathcal{B}^\tau\mathbf{y},\mathbf{y}\right\rangle-i |\mathbf{ t}|\tau_1\right)^{m+n+r-j}} d\tau\\
  &=\int_{-1}^{1}d\tau_1\int_{\sqrt{1-\tau_1^2}S^{r-2}}
  \frac{(\det\mathcal{B}^{(\tau_1,\tau^{\prime\prime})})^{\frac{1}{2}}
  \left(\left\langle\mathcal{B}^{(\tau_1,\tau^{\prime\prime})}\mathbf{y},\mathbf{y}\right\rangle
  +i|\mathbf{t}|\tau_1\right)^{m-j}}
  {\left(\langle\mathcal{B}^{(\tau_1,\tau^{\prime\prime})}\mathbf{y},\mathbf{y}\rangle-i |\mathbf{ t}|\tau_1\right)^{m+n+r-j}}d\tau^{\prime\prime}\\
  &=\int_{-1}^{1}d\tau_1\int_{S^{r-2}}
  \frac{(\det\mathcal{B}^{(\tau_1,\sqrt{1-\tau_1^2}\tau^{\prime})})^{\frac{1}{2}}
  \left(\left\langle\mathcal{B}^{(\tau_1,\sqrt{1-\tau_1^2}\tau^{\prime})}\mathbf{y},\mathbf{y}\right\rangle+i|\mathbf{t}|\tau_1\right)^{m-j}}
  {\left(\left\langle\mathcal{B}^{(\tau_1,\sqrt{1-\tau_1^2}\tau^{\prime})}\mathbf{y},\mathbf{y}\right\rangle-i |\mathbf{ t}|\tau_1\right)^{m+n+r-j}} (1-\tau_1^2)^{\frac{r-2}{2}}d\tau^{\prime}\\
  &=\int_{-1}^{1}d\tau_1\int_{S^{r-2}}
  \frac{(\det\mathcal{B}^{\chi_{\tau^{\prime}}(\tau_1)})^{\frac{1}{2}}
  \left(\left\langle\mathcal{B}^{\chi_{\tau^{\prime}}(\tau_1)}
  \mathbf{y},\mathbf{y}\right\rangle+i|\mathbf{t}|\tau_1\right)^{m-j}(1-\tau_1^2)^{\frac{r-2}{2}}d\tau^{\prime}}
  {\left(\left\langle\mathcal{B}^{\chi_{\tau^{\prime}}(\tau_1)}\mathbf{y},\mathbf{y}\right\rangle-i |\mathbf{ t}|\tau_1\right)^{m+n+r-j}},
    \end{aligned}\end{equation*}
 where  $\tau:=(\tau_1,\tau^{\prime\prime})$ in the second identity,  we take the coordinates transformation $\tau^{\prime\prime}=\sqrt{1-\tau_1^2}\tau^{\prime}$ in the third identity, and for fixed $\tau^{\prime}\in S^{r-2}$,
 define
  \begin{equation}\begin{aligned}\label{chi}
  \chi_{\tau^{\prime}}: & D_{\varepsilon} \rightarrow \Bbb C^r, \\
   z_1 \longmapsto &\left(z_1, \sqrt{1-z_1^2}\tau^{\prime}\right),
  \end{aligned}\end{equation}
where the rectangle  $D_{\varepsilon}:=\{z_1=x+iy\in \Bbb C; x\in(-1,1), |y|<\varepsilon\}$. It is obvious that $\chi_{\tau^{\prime}}(-1,1) \subset S^{r-1}$ for fixed $\tau_0^\prime\in S^{r-2}$. Then $\chi_{\tau^{\prime}}$ is a holomorphic mapping and it is easy to see that $\chi_{\tau^{\prime}}( D_{\varepsilon}) \subset U$ for sufficiently small $\varepsilon$, where $U\subseteq \Bbb C^{r}$ is a neighborhood of $S^{r-1}$, since
   \begin{equation*}
   \left| \left(z_1,\sqrt{1-z_1^2}\tau^{\prime}\right)-\left(x_1,\sqrt{1-x_1^2}\tau^{\prime}\right)\right|\leq C |y|,
    \end{equation*}
    for some constant $C$.
    So the composition $\mathcal{B}^{\chi_{\tau^{\prime}}(z_1)}$ is holomorphic in $ z_1 \in D_{\varepsilon}$ since $\chi_{\tau^{\prime}}(D_{\varepsilon}) \subset U$ and $\mathcal{B}$ is holomorphic in $U$ for fixed $\tau^{\prime}$ by Proposition \ref{pro:B-tau-hol}. Similarly, $(\det\mathcal{B}^{\chi_{\tau^{\prime}}(z_1)})^{\frac{1}{2}}$ is also holomorphic in $z_1$. Denote
  \begin{equation}\label{fmj}
f_{m,j}(\tau^{\prime}; z_1,\mathbf{y}, \mathbf{t}):=\frac{(\det\mathcal{B}^{\chi_{\tau^{\prime}}(z_1)})^{\frac{1}{2}}\left(\langle\mathcal{B}^{\chi_{\tau^{\prime}}(z_1)}
\mathbf{y},\mathbf{y}\rangle+i|\mathbf{ t}|z_1\right)^{m-j}(1-z_1^2)^{\frac{r-2}{2}}}
  {\left(\langle\mathcal{B}^{\chi_{\tau^{\prime}}(z_1)}\mathbf{y},\mathbf{y}\rangle-i|\mathbf{ t}|z_1\right)^{m+n+r-j}}.
  \end{equation}
Let \begin{equation*}\begin{aligned}
    U_{\varepsilon}:=\{z_1=a+ib\in \Bbb C, a\in (-1,1), 0 <b< \varepsilon \sqrt{1-a^2}\}.
    \end{aligned}\end{equation*}
 Since $\mathcal{B}^{\chi_{\tau^{\prime}}(z_1)}$  is holomorphic in $z_1 \in U_{\varepsilon}$, $\mathcal{B}^{\chi_{\tau^{\prime}}(z_1)}=\mathcal{B}^{\chi_{\tau^{\prime}}(a)}+ O(b)$, then
 \begin{equation*}\begin{aligned}
 \left|\langle\mathcal{B}^{\chi_{\tau^{\prime}}(z_1)}\mathbf{y},\mathbf{y}\rangle-i|\mathbf{ t}|z_1\right|
 &= \left|\langle\mathcal{B}^{\chi_{\tau^{\prime}}(a)}\mathbf{y},\mathbf{y}\rangle
 +\langle O(b)\mathbf{y},\mathbf{y}\rangle+|\mathbf{ t}| b-i|\mathbf{ t}|a\right|\\
 &\geq C |\mathbf{y}|^2+|\mathbf{ t}|b-O(b)|\mathbf{y}|^2 >0,
  \end{aligned}\end{equation*}
for some constant $C>0$.

  The numerator of  $f_{m,j}(\tau^{\prime}; z_1,\mathbf{y}, \mathbf{t})$ is holomorphic in $z_1$, while the denominator is still holomorphic in $z_1$  for $\mathbf{y}\neq 0$.
  Hence, $f_{m,j}(\tau^{\prime}; z_1,\mathbf{y}, \mathbf{t})$ is holomorphic in $z_1$ for $\mathbf{y}\neq 0$.
    Note that
    \begin{equation}\label{gmj}
g_{m,j}(z_1; \mathbf{y}, \mathbf{t}):=\int_{S^{r-2}}f_{m,j}(\tau^{\prime}; z_1,\mathbf{y}, \mathbf{t})d\tau^{\prime},
  \end{equation}
  is holomorphic on the domain $U_{\varepsilon}$ and continuous on $\overline{U}_{\varepsilon}$.
  Apply the Cauchy theorem to get
 \begin{equation}\label{eq:ex-I-3}
 \begin{aligned}
  \mathcal{I}_{m,j}(\mathbf{y}, \mathbf{t})&=\int_{L_\varepsilon} g_{m,j}(z_1; \mathbf{y}, \mathbf{t})dz_1,
\end{aligned}\end{equation}
where
     \begin{equation*}\begin{aligned}
    L_{\varepsilon}:=\{z_1=a+ib\in \Bbb C, a\in (-1,1), b= \varepsilon \sqrt{1-a^2}\}.
    \end{aligned}\end{equation*}
    Then it follows from \eqref{eq:ex-I-2}, \eqref{fmj}-\eqref{eq:ex-I-3} that
     \begin{equation}\label{eq:ex-I-4}
 \begin{aligned}
  \mathcal{I}_{m,j}(\mathbf{y}, \mathbf{t})&=\widetilde{\mathcal{I}}_{m,j}(\mathbf{y}, \mathbf{t}),\qquad \text{for}\,\,\mathbf{y}\neq 0,
\end{aligned}\end{equation}
     with $\mathbf{t}=|\mathbf{t}|(1,0,\ldots,0)$ and $O_{\mathbf{t}}$ is the identity matrix.\par
For  $\mathbf{y}=0$, the integral
\begin{equation*}
 \begin{aligned}
\int_{L_\varepsilon} g_{m,j}(z_1; \mathbf{0}, \mathbf{t})dz_1=\int_{L_\varepsilon}dz_1\int_{S^{r-2}}
  \frac{(-1)^{m+n+r-j}(\det\mathcal{B}^{\chi_{\tau^{\prime}}(z_1)})^{\frac{1}{2}}
  (1-z_1^2)^{\frac{r-2}{2}}}
  {\left(i |\mathbf{ t}|z_1\right)^{n+r}}d\tau^{\prime}
  \end{aligned}\end{equation*}
is also well defined since the absolute value of the denominator has a positive lower bound.  Hence the integral in  \eqref{eq:ex-I-4} still make sense for $\mathbf{y}=0, \mathbf{t}\neq0$. \par

For general $\mathbf{t}\in \Bbb R^{r}\backslash \{\mathbf{0}\}$, consider the subspace that orthogonal to $\mathbf{t}$. By Schmidt orthogonalization, we see that there exists  an orthonormal basis $\varepsilon_2,\ldots,\varepsilon_r$ orthogonal to $\mathbf{t}$ which are locally smoothly depends on $\mathbf{t}$.  Let $O_{\mathbf{t}}=(\frac{\mathbf{t}}{|\mathbf{t}|},\varepsilon_2,\ldots,\varepsilon_r)$ and take coordinate changing $\tau=O_{\mathbf{t}}\omega$. Since $O_{\mathbf{t}} \mathbf{t}=|\mathbf{t}|(1,0,\ldots,0)$, then $\mathbf{t} \cdot \tau=|\mathbf{t}|\omega_1$, $d\tau=d\omega$ and
 \eqref{eq:ex-I-1} becomes
\begin{equation}\label{eq:ex-I-5}
 \begin{aligned}
  \mathcal{I}_{m,j}(\mathbf{y}, \mathbf{t})
  =&\int_{S^{r-1}}
  \frac{(\det\mathcal{B}^{O_{\mathbf{t}}\omega})^{\frac{1}{2}}\left(\langle\mathcal{B}^{O_{\mathbf{t}}\omega}\mathbf{y},\mathbf{y}\rangle+i|\mathbf{t}|\omega_1\right)^{m-j}}
  {\left(\langle\mathcal{B}^{O_{\mathbf{t}}\omega}\mathbf{y},\mathbf{y}\rangle-i |\mathbf{t}|\omega_1\right)^{m+n+r-j}} d\omega\\
  =&\int_{-1}^{1}d\omega_1\int_{\sqrt{1-\omega_1^2}S^{r-2}}
  \frac{(\det\mathcal{B}^{O_{\mathbf{t}}(\omega_1,\omega^{\prime\prime})})^{\frac{1}{2}}\left(\langle\mathcal{B}^{O_{\mathbf{t}}(\omega_1,\omega^{\prime\prime})}\mathbf{y},\mathbf{y}\rangle+i|\mathbf{t}|\omega_1\right)^{m-j}}
  {\left(\langle\mathcal{B}^{O_{\mathbf{t}}(\omega_1,\omega^{\prime\prime})}\mathbf{y},\mathbf{y}\rangle-i |\mathbf{ t}|\omega_1\right)^{m+n+r-j}}d\omega^{\prime\prime}\\
  =&\int_{-1}^{1} d\omega_1\int_{S^{r-2}}
  \frac{(\det\mathcal{B}^{O_{\mathbf{t}}(\omega_1,\sqrt{1-\omega_1^2}\omega^{\prime})})^{\frac{1}{2}}\left(\langle\mathcal{B}^{O_{\mathbf{t}}(\omega_1,\sqrt{1-\omega_1^2}\omega^{\prime})}
  \mathbf{y},\mathbf{y}\rangle+i|\mathbf{t}|\omega_1\right)^{m-j}(1-\omega_1^2)^{\frac{r-2}{2}}d\omega^{\prime}}
  {\left(\langle\mathcal{B}^{O_{\mathbf{t}}(\omega_1,\sqrt{1-\omega_1^2}\omega^{\prime})}\mathbf{y},\mathbf{y}\rangle-i |\mathbf{ t}|\omega_1\right)^{m+n+r-j}}\\
  =&\int_{-1}^{1}d\omega_1\int_{S^{r-2}}
  \frac{(\det\mathcal{B}^{O_{\mathbf{t}}\chi_{\omega^{\prime}}(\omega_1)})^{\frac{1}{2}}
  \left(\langle\mathcal{B}^{O_{\mathbf{t}}\chi_{\omega^{\prime}}(\omega_1)}
  \mathbf{y},\mathbf{y}\rangle+i|\mathbf{t}|\omega_1\right)^{m-j}(1-\omega_1^2)^{\frac{r-2}{2}}d\omega^{\prime}}
  {\left(\langle\mathcal{B}^{O_{\mathbf{t}}\chi_{\omega^{\prime}}(\omega_1)}\mathbf{y},\mathbf{y}\rangle-i |\mathbf{ t}|\omega_1\right)^{m+n+r-j}},
  \end{aligned}\end{equation}
where $\omega=(\omega_1,\omega^{\prime\prime})$ and we take the coordinates transformation $\omega^{\prime\prime}=\sqrt{1-\omega_1^2}\omega^{\prime}$ in the last identity.  $\langle\mathcal{B}^{O_{\mathbf{t}}\omega}\mathbf{y},\mathbf{y}\rangle$ and $(\det\mathcal{B}^{O_{\mathbf{t}}\omega})^{\frac{1}{2}}$  are  locally smoothly dependent on $\mathbf{t}$. Then the possible singular points of integral \eqref{eq:ex-I-5}   only appear in the denominator, so it is similar to the case $\mathbf{t}=|\mathbf{t}|(1,0,\ldots,0)$. We can also change the contour integral to get \eqref{eq:ex-I-4} holds for $\mathbf{y}\neq0$ and \eqref{eq:ex-I-2} still make sense for $\mathbf{y}=0, \mathbf{t}\neq0$. We omit details.
The theorem is proved.
\end{proof}
\begin{rem} For H-type groups, since $\mathcal{B}^{\tau}=I_{2n}$, it is not necessary to extend $\mathcal{B}^{\tau}$ to a holomorphic function and use orthogonal matrix $O_{\mathbf{t}}$. We can change contour of $(-1,1)$ directly to $L_{\varepsilon}$ to obtain
  \begin{equation*}
 \begin{aligned}
  \widetilde{\mathcal{I}}_{m,j}(\mathbf{y}, \mathbf{t})
  :=\int_{L_\varepsilon}dz\int_{S^{r-2}}
  \frac{
  \left(\langle  \mathbf{y},\mathbf{y}\rangle+i|\mathbf{t}|z\right)^{m-j}(1-z^2)^{\frac{r-2}{2}}d\tau^{\prime}}
  {\left(\langle\mathbf{y},\mathbf{y}\rangle-i |\mathbf{ t}|z\right)^{m+n+r-j}}.
  \end{aligned}\end{equation*}
\end{rem}

\section{The homogeneous distribution $\mathscr{P}_m$ }\label{section4}
\subsection{The distribution $\mathscr{P}_m$}
Recall that a distribution $\mathscr K$ on $\mathcal{N}$ is called {\it homogeneous of degree $a$} if
\begin{equation}\label{def:homogeneous}
\begin{aligned}
\left\langle \mathscr K,\phi_{\lambda}\right\rangle=&\lambda^{a} \left\langle \mathscr K,\phi\right\rangle,\,\,\,\,\,\,\text{for any}\,\, \phi\in \mathcal{S}(\mathcal{N}),
\end{aligned}
\end{equation}
where $\phi_\lambda(\mathbf{y}, \mathbf{t})=\lambda^{-Q} \phi\left(\lambda^{-1} \mathbf{y}, \lambda^{-2} \mathbf{t}\right) $. Let $\delta_\lambda(\mathbf{y},\mathbf{t})=(\lambda \mathbf{y},\lambda^2 \mathbf{t})$. Then for
  a homogeneous distribution $\mathscr{K}$ of degree $-Q$, we have
  \begin{equation}\label{homo-2}
\left\langle \mathscr{K}, \phi\circ\delta_{\lambda}\right\rangle=\left\langle \mathscr{K},\phi\right\rangle.
\end{equation}
\begin{pro}\label{pm-homo}
   $\mathscr{P}_{m}$ defined by \eqref{def-dis-pm} is a  distribution. Moreover, it is homogeneous of degree $-Q$.
\end{pro}
\begin{proof}
 Since $\|\widetilde{\phi}\|_{(N,0)} \leq C_N \|\phi\|_{(N,2N)}$, for some constant $C_N$ depending on $N$, by definition
$$\|\phi\|_{(k, l)}:= \sup_{(\mathbf{y}, \mathbf{t})\in \Bbb R^{2n+r}, |\mathbf{I}|=l}\left(1+|\mathbf{y}|^4+|\mathbf{t}|^2\right)^{k}|\partial^{\mathbf{I}}_{(\mathbf{y},\mathbf{t})}\phi(\mathbf{y}, \mathbf{t})|, $$
where $\partial^{\mathbf{I}}_{(\mathbf{y},\mathbf{t})}=\partial^{I_1}_{y_1}\cdots\partial^{I_{2n}}_{y_{2n}}\partial^{I_{2n+1}}_{t_1}\cdots\partial^{I_{2n+r}}_{y_r}$ and $|\mathbf{I}|=I_1+\cdots+I_{2n+r}$, it follows from the definition \eqref{def-dis-pm} of $\mathscr{P}_{m}$ and the estimate \eqref{Qm-integrable} that
\begin{equation*}\label{pm-bounded}
\begin{aligned}
\left|\left\langle\mathscr{P}_{m},\phi\right\rangle\right|&\leqslant\frac{1}{(2\pi)^r} \int_{\Bbb R^{2n+r}} |Q_m(\mathbf{y},\tau )||\widetilde{\phi}(\mathbf{y},\tau )| d\mathbf{y} d \tau \\
 &\leqslant \frac{1}{(2\pi)^r}\int_{\Bbb R^{2n+r}} |Q_m(\mathbf{y},\tau )|\frac{\|\widetilde{\phi}\|_{(N,0)}}{(1+|\mathbf{y}|^4+|\tau|^2)^{N}} d\mathbf{y} d \tau \\
& \leq C_m\|\widetilde{\phi}\|_{(N, 0)}\\
& \leq C_m C_N\|\phi\|_{(N, 2N)},
\end{aligned}
\end{equation*}
for $\phi\in \mathcal{S}(\mathcal{N})$ and sufficiently large $N$, where  $C_m$ is a constant depending on $m$. Hence,
 $\mathscr{P}_m$ is a distribution.

 Since $\phi_\lambda(\mathbf{y}, \mathbf{t})=\lambda^{-Q} \phi\left(\lambda^{-1} \mathbf{y}, \lambda^{-2} \mathbf{t}\right) $, then
\begin{equation*}
\begin{aligned}
\widetilde{\phi}_\lambda(\mathbf{y}, -\tau)
 =\lambda^{-Q} \int_{\Bbb R^r} \phi\left(\lambda^{-1} \mathbf{y}, \lambda^{-2} \mathbf{t}\right) e^{i \mathbf{t} \cdot \tau} d \mathbf{t} =\lambda^{-2 n} \widetilde{\phi}\left(\lambda^{-1} \mathbf{y}, -\lambda^{2} \tau\right).
\end{aligned}
\end{equation*}
Hence,
\begin{equation*}
\begin{aligned}
\left\langle\mathscr{P}_{m},\phi_{\lambda}\right\rangle=&\frac{1}{(2\pi)^r}\int_{\Bbb R^{2n+r}}Q_m(\mathbf{y},\tau)\widetilde{\phi}_{\lambda}(\mathbf{y},-\tau) d\mathbf{y}d\tau\\
=&\frac{1}{(2\pi)^r}\cdot\frac{2^{n}}{\pi^{n}}\int_{\Bbb R^{2n+r}} (\det\mathcal{B}^\tau)^{\frac{1}{2}}e^{-\langle\mathcal{B}^\tau\mathbf{y},\mathbf{y}\rangle } L_{m}^{(n-1)}(2\langle\mathcal{B}^\tau\mathbf{y},\mathbf{y}\rangle )\lambda^{-2 n} \widetilde{\phi}\left(\lambda^{-1} \mathbf{y}, -\lambda^{2} \tau\right)d\mathbf{y}d\tau\\
=&\frac{\lambda^{-Q}}{(2\pi)^r} \int_{\Bbb R^{2n+r}}Q_m(\mathbf{y}^{\prime},\tau^{\prime})\widetilde{\phi}(\mathbf{y}^{\prime},-\tau^{\prime}) d\mathbf{y}^{\prime}d\tau^{\prime}\\
=&\lambda^{-Q} \left\langle\mathscr{P}_{m},\phi\right\rangle,
\end{aligned}
\end{equation*}
by taking the coordinates transformation $\mathbf{y}^\prime=\lambda^{-1}\mathbf{y}, \tau^{\prime}=\lambda^{2}\tau$ in the third identity. Here $\langle\mathcal{B}^\tau\mathbf{y},\mathbf{y}\rangle$ is invariant under this transformation and $(\det\mathcal{B}^\tau)^{\frac{1}{2}}=\lambda^{-2n}(\det\mathcal{B}^{\tau^{\prime}})^{\frac{1}{2}}$.
So $\mathscr{P}_m$ is homogeneous of degree $-Q$ by \eqref{def:homogeneous}.
\end{proof}

\begin{pro}\label{pro:pm-coincide-pm}
Suppose that $\mathcal{N}$ is a non-degenerate nilpotent Lie group of step two. Then the  distribution $\mathscr{P}_m$ defined by \eqref{def-dis-pm} coincides with the Lipschitzian function $P_m$  on $\mathcal{N}\setminus \{\mathbf{0}\}$.
\end{pro}
\begin{proof}
For $\phi\in S(\mathcal{N}\setminus \{\mathbf{0}\})$, there exists $\varepsilon_0 > 0$ such that $supp  \phi \subseteq  \{(\mathbf{y},\mathbf{t}): |(\mathbf{y},\mathbf{t})|>\epsilon_{0}\}$.
Since $\phi$ is a Schwartz function on variable $(\mathbf{y}, \mathbf{t})$, $\widetilde{\phi}$ also decays fast. $L^{(n-1)}_m$ is a polynomial of degree $m$, then it follows from the expression \eqref{Qm-2} and Lemma \ref{lem:B-y} that for any given $N\in \Bbb N$, there exists a constant $C_N>0$, such that
\begin{equation}\label{esti-I1}
\begin{aligned}
|Q_m(\mathbf{y},\tau)\widetilde{\phi}(\mathbf{y},-\tau)|\leq  \frac{C_N|\tau|^{n}(|\mathbf{\tau}||\mathbf{y}|^{2}+1)^m}{1+(|\mathbf{y}|^{4}+|\mathbf{\tau}|^2)^{N}},
\end{aligned}
\end{equation}
almost everywhere. \par
If $ |\mathbf{y}|\neq 0$, since $B^\tau$ is non degenerate, it follows from  Lemma \ref{lem:B-y} that there exists a constant $C>0$, such that
\begin{equation*}\label{esti-I2}
\begin{aligned}
&\left|Q_m(\mathbf{y},\tau)\phi(\mathbf{y},\mathbf{t})e^{i\mathbf{t}\cdot \tau}\right|
\lesssim e^{-C^{-1}|\tau||\mathbf{y}|^{2}}\frac{|\tau|^{n}(|\mathbf{\tau}||\mathbf{y}|^{2}+1)^m}{1+(|\mathbf{y}|^{4}+|\mathbf{t}|^2)^{N}},
\end{aligned}
\end{equation*}
for sufficient large $N$ almost everywhere on $\Bbb R^{2n}\times \Bbb R^{r}\times\Bbb R^{r}$. Consequently, it is  integrable in $(\mathbf{y},\mathbf{t},\tau)\in B_\varepsilon^c \times \Bbb R^r \times \Bbb R^r$, where $B_\varepsilon^c=\{\mathbf{y}: |\mathbf{y}|>\varepsilon\}$.  By using Fubini's theorem, we get
\begin{equation*}\label{change-inte}
\begin{aligned}
\int_{B_\varepsilon^c\times \Bbb R^r}\phi(\mathbf{y},\mathbf{t})P_m(\mathbf{y},\mathbf{t}) d\mathbf{y}d\mathbf{t}=&
\frac{1}{(2\pi)^r}\int_{B_\varepsilon^c\times \Bbb R^r}\phi(\mathbf{y},\mathbf{t})d\mathbf{t}\int_{\Bbb R^r}Q_m(\mathbf{y},\tau)e^{i\mathbf{t}\cdot \tau} d\mathbf{y}d\tau\\
=&\frac{1}{(2\pi)^r}\int_{B_\varepsilon^c\times \Bbb R^r}Q_m(\mathbf{y},\tau)d\tau\int_{\Bbb R^r}\phi(\mathbf{y},\mathbf{t})e^{i\mathbf{t}\cdot \tau}  d\mathbf{y} d\mathbf{t}\\
=&\frac{1}{(2\pi)^r}\int_{B_\varepsilon^c\times \Bbb R^r}Q_m(\mathbf{y},\tau)\widetilde{\phi}(\mathbf{y},-\tau) d\mathbf{y}d\tau,
\end{aligned}
\end{equation*}
by using \eqref{Pm-Qm} in the first identity. Since $P_m$ is Lipschitzian on $\mathcal{N}\backslash \{\mathbf{0}\}$ by its expression \eqref{eq:ex-pm-2}-\eqref{eq:ex-I-2} in Theorem \ref{th:pm-ex}, we see that
\begin{equation*}
\begin{aligned}
\int_{\Bbb R^{2n+r}}\phi(\mathbf{y},\mathbf{t})P_m(\mathbf{y},\mathbf{t}) d\mathbf{y}d\mathbf{t}=&\lim_{\varepsilon\rightarrow 0}
\int_{B_\varepsilon^c\times \Bbb R^r}\phi(\mathbf{y},\mathbf{t})P_m(\mathbf{y},\mathbf{t}) d\mathbf{y} d\mathbf{t}\\
=&\lim_{\varepsilon\rightarrow 0}\frac{1}{(2\pi)^r}\int_{B_\varepsilon^c\times \Bbb R^r}Q_m(\mathbf{y},\tau)\widetilde{\phi}(\mathbf{y},-\tau) d\mathbf{y}d\tau\\
=&\frac{1}{(2\pi)^r}\int_{\Bbb R^{2n+r}}Q_m(\mathbf{y},\tau)\widetilde{\phi}(\mathbf{y},-\tau) d\mathbf{y}d\tau\\
=&\left\langle\mathscr{P}_m,\phi\right\rangle.
\end{aligned}
\end{equation*}
The third identity follows from $Q_m\widetilde{\phi}\in L^1(\mathcal{N})$ by \eqref{esti-I1} and $\phi\in S(\mathcal{N}\setminus \{\mathbf{0}\})$.
 So $\mathscr{P}_m$ coincides with the  function $P_m$ on $\mathcal{N}\setminus \{\mathbf{0}\}$.
\end{proof}

We need the following characterization of  homogeneous distribution on a homogeneous group given by Folland-Stein \cite{FS1982}.

\begin{pro}\label{expre-dis}\cite[Proposition 6.13]{FS1982}
  Let $K$ be a continuous homogeneous function of degree $-Q$ on a homogeneous group  $G \setminus \{\mathbf{0}\}$ and satisfies
  \begin{equation}\label{mu-k}
  \int_{\partial B(\mathbf{0},1)}K(\mathbf{y},\mathbf{t})d\mathbf{y}d\mathbf{t}=0.
\end{equation}
  Then the formula
\begin{equation*}
  \langle p.v. K, \phi\rangle=\lim_{\epsilon\rightarrow 0}\int_{\Bbb R^{2n+r}\setminus B(\mathbf{0},\epsilon)}K(\mathbf{y},\mathbf{t})\phi(\mathbf{y},\mathbf{t})d\mathbf{y}d\mathbf{t}
\end{equation*}
for any $\phi\in \mathscr{D}(G)$, defines a tempered distribution $p.v. K$ which is homogeneous of degree $-Q$.\par Conversely, suppose $\mathscr{K}$ is a tempered distribution which is homogeneous of degree $-Q$ and whose restriction to $G \setminus \{\mathbf{0}\}$ is a continuous function $K$. Then \eqref{mu-k} holds and
\begin{equation*}
 \mathscr{K} = p.v. K+C\delta_{\mathbf{0}},
\end{equation*}
for some constant $C$, where $\delta_{\mathbf{0}}$ is the delta function  at $\mathbf{0}$.
\end{pro}

\emph{Proof of Theorem \ref{pf=f+pvf}.}  By Theorem \ref{th:pm-ex}, $P_m$ is a continuous homogeneous function of degree $-Q$ on $\mathcal{N}\setminus \{\mathbf{0}\}$. By Proposition \ref{pm-homo} and Proposition \ref{pro:pm-coincide-pm}, the distribution $\mathscr{P}_m$ is homogeneous of degree $-Q$ and coincide with  Lipschitzian function  $P_m$ on  $\mathcal{N}\backslash \{\mathbf{0}\}$.
Hence, the distribution $\mathscr{P}_m$ has the expression \eqref{eq:pf=f+pvf} and $P_m$ satisfies \eqref{mu-k=0} by Proposition \ref{expre-dis}.
The theorem is proved.\qed
\vskip 5mm

\subsection{$P_m$ is a  Calder\'{o}n-Zygmund kernel}
For a fixed point $(\mathbf{x}, \mathbf{t}) \in \mathcal{N}$, the left multiplication by $(\mathbf{x}, \mathbf{t})$ is an affine transformation of $\mathbb{R}^{2 n+r}$:
$$
\mathbf{y} \mapsto \mathbf{y}+\mathbf{x}, \quad \mathbf{s} \mapsto \mathbf{s}+\mathbf{t}+2 B(\mathbf{x}, \mathbf{y}),
$$
which preserves the Lebegues measure $d \mathbf{y} d \mathbf{s}$ of $\mathbb{R}^{2 n+r}$. The measure $d \mathbf{y} d \mathbf{s}$ is also right invariant, and so it is a Haar measure on the nilpotent Lie group $\mathcal{N}$ of step two. The convolution on $\mathcal{N}$ is defined as
\begin{equation*}
\varphi * \psi(\mathbf{y}, \mathbf{s}):=\int_{\mathcal{N}} \varphi(\mathbf{x}, \mathbf{t}) \psi\left((\mathbf{x}, \mathbf{t})^{-1}(\mathbf{y}, \mathbf{s})\right) d \mathbf{x} d \mathbf{t}=\int_{\mathcal{N}} \varphi\left((\mathbf{y}, \mathbf{s})(\mathbf{x}, \mathbf{t})^{-1}\right) \psi(\mathbf{x}, \mathbf{t}) d \mathbf{x} d \mathbf{t},
\end{equation*}
for $\phi, \psi \in L^1(\mathcal{N})$.
For $\phi\in \mathcal{S}(\mathcal{N})$, the convolution of $\phi$ and the distribution $\mathscr{P}_m$ is defined as
\begin{equation*}\label{def-con-pm}
\phi \ast\mathscr{P}_m(\mathbf{y},\mathbf{t}):=\left\langle \mathscr{P}_m, \phi_{(\mathbf{y},\mathbf{t})}\right\rangle,
\end{equation*}
where $\phi_{(\mathbf{y},\mathbf{t})}(\mathbf{y^{\prime}},\mathbf{t^{\prime}})=\phi((\mathbf{y},\mathbf{t})(\mathbf{y^{\prime}},\mathbf{t^{\prime}})^{-1})$.
Let $\phi_1(\mathbf{y^{\prime}},\mathbf{t^{\prime}})=\phi_{(\mathbf{y},\mathbf{t})}(\mathbf{y^{\prime}},-\mathbf{t^{\prime}})$. By the definition of the distribution $\mathscr P_m$, we have
\begin{equation*}\begin{aligned}
\phi\ast\mathscr{P}_m(\mathbf{y},\mathbf{t})=&\frac{1}{(2\pi)^r}\int_{\Bbb R^{2n+r}} Q_m(\mathbf{y^{\prime}},\tau)\widetilde{\phi}_{(\mathbf{y},\mathbf{t})}(\mathbf{y^{\prime}},-\tau)d\mathbf{y^{\prime}}d\tau\\
=&\frac{1}{(2\pi)^r}\int_{\Bbb R^{2n+r}} Q_m(\mathbf{y^{\prime}},\tau)\widetilde{\phi}_{1}(\mathbf{y^{\prime}},\tau)d\mathbf{y^{\prime}}d\tau,
\end{aligned}\end{equation*}
since
$$\phi_1(\mathbf{y^{\prime}},\mathbf{t^{\prime}})
=\phi((\mathbf{y},\mathbf{t})(\mathbf{y^{\prime}},-\mathbf{t^{\prime}})^{-1})
=\phi((\mathbf{y},\mathbf{t})(-\mathbf{y^{\prime}},\mathbf{t^{\prime}})),$$
then
\begin{equation*}\label{tilde-phi1}\begin{aligned}
\widetilde{\phi}_1(\mathbf{y^{\prime}},\tau)&= \int_{\Bbb R^{r}} \phi(\mathbf{y}-\mathbf{y^{\prime}}, \mathbf{t}+\mathbf{t}^{\prime}-2B(\mathbf{y},\mathbf{y^{\prime}}))e^{-i\mathbf{t}^{\prime} \cdot \tau}d\mathbf{t}^{\prime}\\
&=\int_{\Bbb R^{r}} \phi(\mathbf{y}-\mathbf{y^{\prime}}, \mathbf{u})e^{-i(\mathbf{u}-\mathbf{t}+2B(\mathbf{y},\mathbf{y^{\prime}})) \cdot \tau}d\mathbf{u}\\
&=\widetilde{\phi}(\mathbf{y}-\mathbf{y^{\prime}}, \tau)e^{i\mathbf{t}\cdot \tau-2iB(\mathbf{y},\mathbf{y^{\prime}}) \cdot \tau},
\end{aligned}\end{equation*}
by taking transformation $\mathbf{u}=\mathbf{t}+\mathbf{t}^{\prime}-2B(\mathbf{y},\mathbf{y^{\prime}})$ in the second identity and
 \begin{equation}\label{con-pm-2}\begin{aligned}
\widetilde{\phi}_{(\mathbf{y},\mathbf{t})}(\mathbf{y^{\prime}},-\tau)&=\int_{\Bbb R^{r}} \phi_{(\mathbf{y},\mathbf{t})}(\mathbf{y^{\prime}},\mathbf{t^{\prime}})e^{i\mathbf{t}^{\prime} \cdot \tau}d\mathbf{t}^{\prime}\\
&=\int_{\Bbb R^{r}} \phi(\mathbf{y}-\mathbf{y^{\prime}}, \mathbf{t}-\mathbf{t}^{\prime}-2B(\mathbf{y},\mathbf{y^{\prime}}))e^{i\mathbf{t}^{\prime} \cdot \tau}d\mathbf{t}^{\prime}\\
&=\int_{\Bbb R^{r}} \phi(\mathbf{y}-\mathbf{y^{\prime}}, \mathbf{u})e^{i(\mathbf{t}-\mathbf{u}-2B(\mathbf{y},\mathbf{y^{\prime}})) \cdot \tau}d\mathbf{u}\\
&=\widetilde{\phi}(\mathbf{y}-\mathbf{y^{\prime}}, \tau)e^{i\mathbf{t}\cdot \tau-2iB(\mathbf{y},\mathbf{y^{\prime}}) \cdot \tau}=\widetilde{\phi}_1(\mathbf{y^{\prime}},\tau),
\end{aligned}\end{equation}
where we  take transformation $\mathbf{u}=\mathbf{t}-\mathbf{t}^{\prime}-2B(\mathbf{y},\mathbf{y^{\prime}})$ in the third identity.

Recall that an absolutely continuous curve $\gamma:[0,1] \rightarrow \mathcal{N}$ is horizontal if its tangent vectors $\dot{\gamma}(s), s \in[0,1]$, lie
in the horizontal tangent space $H_{\gamma(s)}$. Any given two points $\mathbf{p}, \mathbf{q} \in \mathcal{N}$ can be connected by a horizontal curve.
The Carnot-Carath\'{e}odory metric on $\mathcal{N}$ is defined as follows:
\begin{equation}\label{def:dcc}
d_{c c}(\mathbf{g},\mathbf{ h}):=\inf _\gamma \int_0^1\langle\dot{\gamma}(s), \dot{\gamma}(s)\rangle_H^{1 / 2} \mathrm{~d} s \quad \text { for } \mathbf{g}, \mathbf{h} \in \mathcal{N},
\end{equation}
where $\gamma:[0,1] \rightarrow \mathcal{N}$ is a horizontal Lipschitzian curve with $\gamma(0)=\mathbf{g}, \gamma(1)=\mathbf{h}$ and $\langle \cdot, \cdot\rangle_H$ is the standard Hermitian inner product on $\Bbb C^{2n}$.
In order to  proving  Theorem \ref{thm:con-Pm-bounded}, we need a mean value theorem on stratified
 groups. Note that there is a version  \cite[(1.41)]{FS1982} for $C^1$ functions, but  the kernel $P_m$ of the spectrum projection operator $\mathbb{P}_m$ is only Lipschitzian and it has singularity at the origin.
  We now need the following mean value theorem on $\mathcal{N}$ adapted to our kernel $P_m$.
\begin{lem}\label{lem:mean value} Let $K$ be a Lipschitzian function on $\mathcal{N} \backslash\{\mathbf{0}\}$ and $Y_j$'s are left invariant vector fields on $\mathcal{N}$ defined by \eqref{eq:Y}. There are $C>0$ and $c_0 \in(0,1)$ such that for $\mathbf{g}, \mathbf{g}_0 \in \mathcal{N} \backslash\{\mathbf{0}\}$ with $d_{c c}\left(\mathbf{g}, \mathbf{g}_0\right)<c_0 d_{c c}\left(\mathbf{g}_0, \mathbf{0}\right)$, we have
$$
\left|K(\mathbf{g})-K\left(\mathbf{g}_0\right)\right| \leq C d_{c c}\left(\mathbf{g}, \mathbf{g}_0\right) \times \max _{\substack{1 \leq j \leq 2n \\ \mathbf{u}:\, d_{c c}(\mathbf{u}, \mathbf{0}) \leq d_{c c}\left(\mathbf{g},\mathbf{g}_0\right)}}\left|Y_j K\left(\mathbf{g}_0 \cdot \mathbf{u}\right)\right|.
$$
\end{lem}

For any $\mathbf{g}=(\mathbf{y}, \mathbf{t}) \in \mathcal{N}$, the homogeneous norm of $\mathbf{g}$ is defined by
$$\|\mathbf{g}\|=\left(|\mathbf{y}|^4+|\mathbf{t}|^2\right)^{1 / 4}.$$
Obviously, $\left\|\mathbf{g}^{-1}\right\|=\|\mathbf{g}\|$ and $\left\|\delta_\lambda(\mathbf{g})\right\|=\lambda\|\mathbf{g}\|$, where $\delta_\lambda, \lambda>0$, is the dilation on $\mathcal{N}$, which is defined as $\delta_\lambda( \mathbf{y}, \mathbf{t})=\left(\lambda\mathbf{y}, \lambda^2\mathbf{t}\right)$.
On $\mathcal{N}$, we define the quasi-distance
$$\rho(\mathbf{h}, \mathbf{g})=\left\|\mathbf{g}^{-1} \cdot\mathbf{ h}\right\|.$$
It is clear that $\rho$ is symmetric and satisfies the generalized triangle inequality
\begin{equation}\label{tri-inequ}
\rho(\mathbf{h}, \mathbf{g}) \leq C_\rho(\rho(\mathbf{h}, \mathbf{w})+\rho(\mathbf{w}, \mathbf{g})),
\end{equation}
for any $\mathbf{h}, \mathbf{g}, \mathbf{w} \in \mathcal{N}$ and some $C_\rho>0$.

\vskip 5mm
\emph{Proof of Theorem \ref{thm:con-Pm-bounded}.} Recall that the kernel $P_m$ is Lipschitzian on $\mathcal{N}\setminus \{\mathbf{0}\}$, and is homogeneous of degree $-Q$ by  Theorem \ref{th:pm-ex}.
 In order to proving $P_m$ is a  Calder\'{o}n-Zygmund kernel, we need to prove kernel $P_m(\mathbf{g},\mathbf{h})$ on $\mathcal{N}$ for $\mathbf{g} \neq \mathbf{h}, \mathbf{g},\mathbf{h} \in \mathcal{N}$ satisfies the following conditions:\\
(i) $\quad|P_m(\mathbf{h}^{-1}\cdot\mathbf{g})| \lesssim  \frac{1}{\rho(\mathbf{g}, \mathbf{h})^Q}$;\\
(ii) $\quad\left|P_m(\mathbf{h}^{-1}\cdot\mathbf{g})-P_m\left(\mathbf{h}^{-1}\cdot\mathbf{g}_0\right)\right| \lesssim  \frac{\rho\left(\mathbf{g}, \mathbf{g}_0\right)}{\rho\left(\mathbf{g}_0, \mathbf{h}\right)^{Q+1}}, \quad$ if $\rho\left(\mathbf{g}_0,\mathbf{h}\right) \geq c \rho\left(\mathbf{g}, \mathbf{g}_0\right)$;\\
(iii) $\quad\left|P_m(\mathbf{h}^{-1}\cdot\mathbf{g})-P_m\left(\mathbf{h}^{-1}_0 \cdot\mathbf{g}\right)\right| \lesssim \frac{\rho\left(\mathbf{h}, \mathbf{h}_0\right)}{\rho\left(\mathbf{g}, \mathbf{h}_0\right)^{Q+1}}, \quad$ if $\rho\left(\mathbf{g}, \mathbf{h}_0\right) \geq c \rho\left(\mathbf{h}, \mathbf{h}_0\right)$
for some constant $c>0$, where $Q=2n+2r$ is the homogeneous dimension of $\mathcal{N}$.\par
Note that
 \begin{equation*}\begin{aligned}
 \mathcal{I}_{m,j}(-\mathbf{y}, -\mathbf{t})=&\int_{S^{r-1}}
  \frac{(\det\mathcal{B}^\tau)^{\frac{1}{2}}\left(\langle\mathcal{B}^\tau\mathbf{y},\mathbf{y}\rangle+i(-\mathbf{ t}) \cdot \tau\right)^{m-j}}
  {\left(\langle\mathcal{B}^\tau\mathbf{y},\mathbf{y}\rangle-i (-\mathbf{t })\cdot \tau\right)^{m+n+r-j}} d\tau
  =\overline{\mathcal{I}_{m,j}(\mathbf{y}, \mathbf{t})}.
  \end{aligned}\end{equation*}
 It follows from the integral representation formula \eqref{eq:ex-pm}-\eqref{eq:ex-I-1}  that
\begin{equation}\label{p-y-t=pyt}
P_m(-\mathbf{y},-\mathbf{t})=\overline{P_m(\mathbf{y},\mathbf{t})}.
\end{equation}
So we only need to prove (i) and (ii). Denote
\begin{equation*}
\mathscr{B}(\mathbf{y},\mathbf{t},\tau)=\langle\mathcal{B}^\tau\mathbf{y},\mathbf{y}\rangle-i \mathbf{ t}\cdot\tau.
\end{equation*}
 We begin with proving the size estimate  (i).   For $|\mathbf{y}|^2  \geq \varepsilon^2 |\mathbf{t}|$ with $\varepsilon$ as in the proof of Theorem \ref{th:pm-ex},  there exists a constant $C>1$ such that
$$\left|\overline{\mathscr{B}(\mathbf{y},\mathbf{t},\tau)}\right|^2\leq C|\mathbf{y}|^4+|\mathbf{t}|^2\leq C|\mathbf{y}|^4+C|\mathbf{t}|^2=C\|\mathbf{g}\|^4,$$
by Lemma \ref{lem:B-y} and
$$\left|\mathscr{B}(\mathbf{y},\mathbf{t},\tau)\right|^2\geq C^{-1}|\mathbf{y}|^4\gtrsim |\mathbf{y}|^4+|\mathbf{t}|^2 =\|\mathbf{g}\|^4, \qquad\text{for}\,\,\tau\in S^{r-1}.$$
Therefore
\begin{equation*}\label{modu-a-b-1}
\begin{aligned}
\left|\frac{ \overline{\mathscr{B}(\mathbf{y},\mathbf{t},\tau)}^{a}}{\mathscr{B}(\mathbf{y},\mathbf{t},\tau)^{b}}\right|
\leq\frac{C}{\|\mathbf{g}\|^{2b-2a}},
\end{aligned}
\end{equation*}
for $a,b\in \Bbb N_0$. Similarly, for $|\mathbf{y}|^2 < \varepsilon^2 |\mathbf{t}|$,
$$\left|\langle\mathcal{B}^{O_{\mathbf{t}}\chi_{\omega^{\prime}}(z)}\mathbf{y},\mathbf{y}\rangle+i |\mathbf{ t}|z\right|^2\leq C|\mathbf{y}|^4+|\mathbf{t}|^2\leq C|\mathbf{y}|^4+C|\mathbf{t}|^2=C\|\mathbf{g}\|^4,$$
and
$$\left|\langle\mathcal{B}^{O_{\mathbf{t}}\chi_{\omega^{\prime}}(z)}\mathbf{y},\mathbf{y}\rangle-i |\mathbf{ t}|z\right|\geq \varepsilon|\mathbf{t}|-C \varepsilon^2 |\mathbf{t}| \gtrsim \|\mathbf{g}\|^2,$$
since $\varepsilon\leq |z|\leq 1$ for $z\in L_\varepsilon$. Here we take sufficiently small $\varepsilon$ such that $\frac{1}{\varepsilon}> C$.
Therefore
\begin{equation*}\label{modu-a-b-2}
\begin{aligned}
\left|\frac{ \left(\langle\mathcal{B}^{O_{\mathbf{t}}\chi_{\omega^{\prime}}(z)}\mathbf{y},\mathbf{y}\rangle+i |\mathbf{ t}|z\right)^{a}}{\left(\left|\langle\mathcal{B}^{O_{\mathbf{t}}\chi_{\omega^{\prime}}(z)}\mathbf{y},\mathbf{y}\rangle-i |\mathbf{ t}|z\right|\right)^{b}}\right|
\leq\frac{C}{\|\mathbf{g}\|^{2b-2a}},
\end{aligned}
\end{equation*}
for $a,b\in \Bbb N_0$.
By Proposition \ref{pro:pm-expression} and Theorem \ref{th:pm-ex}, we have
\begin{equation*}
 \begin{aligned}
  \left|P_{m}(\mathbf{g})\right|\leq\frac{C}{\|\mathbf{g}\|^Q},
  \end{aligned}\end{equation*}
 where the implicit constant depends on $m, n, r$,  which shows that (i) holds.\par
 We now prove the regularity estimate (ii).
For $\mathbf{y} \neq \mathbf{0}$,
 \begin{equation*} \begin{aligned}
  Y_jP_{m}(\mathbf{g}) = \sum_{j=0}^{r}C_{m,j}\int_{S^{r-1}} (\det\mathcal{B}^\tau)^{\frac{1}{2}}\cdot Y_j\frac{\overline{\mathscr{B}(\mathbf{y},\mathbf{t},\tau)}^{m-j}}{\mathscr{B}(\mathbf{y},\mathbf{t},\tau)^{m+n+r-j}} d\tau,
  \end{aligned}\end{equation*}
 by Proposition \ref{pro:pm-expression}.  Note that
 \begin{equation}\label{Y-Pm-a-b-1}\begin{aligned}
Y_j\frac{\overline{\mathscr{B}(\mathbf{y},\mathbf{t},\tau)}^{a}}{\mathscr{B}(\mathbf{y},\mathbf{t},\tau)^{b}}
=&\frac{aM_j\overline{\mathscr{B}(\mathbf{y},\mathbf{t},\tau)}^{a-1}\mathscr{B}(\mathbf{y},\mathbf{t},\tau)^{b}}
{\mathscr{B}(\mathbf{y},\mathbf{t},\tau)^{2b}}
-\frac{b\overline{M}_j\overline{\mathscr{B}(\mathbf{y},\mathbf{t},\tau)}^{a}\mathscr{B}(\mathbf{y},\mathbf{t},\tau)^{b-1}}
{\mathscr{B}(\mathbf{y},\mathbf{t},\tau)^{2b}},
\end{aligned}\end{equation}
 where $ M_j:=Y_j\overline{\mathscr{B}(\mathbf{y},\mathbf{t},\tau)}.$
Since
 \begin{equation*}
 \left|\frac{\partial}{\partial y_j} \mathscr{B}(\mathbf{y},\mathbf{t},\tau)\right| \leq C\left(|\mathbf{y}|^4+|\mathbf{t}|^{2}\right)^{\frac{1}{4}}
\quad\text{and}\,\,
 \left|\frac{\partial}{\partial t_j} \mathscr{B}(\mathbf{y},\mathbf{t},\tau)\right| =|\tau_j|\leq  |\tau|=1,
 \end{equation*}
 we have
  \begin{equation}\label{Mj}
 \left|M_j\right| \leq C\left(|\mathbf{y}|^4+|\mathbf{t}|^{2}\right)^{\frac{1}{4}}.
 \end{equation}
 Hence,
 \begin{equation*}\label{modu-I1}\begin{aligned}
\left|Y_j\frac{\overline{\mathscr{B}(\mathbf{y},\mathbf{t},\tau)}^{a}}{\mathscr{B}(\mathbf{y},\mathbf{t},\tau)^{b}}\right|
\leq C\left(|\mathbf{y}|^4+|\mathbf{t}|^{2}\right)^{\frac{1}{4}} \frac{\left|\overline{\mathscr{B}(\mathbf{y},\mathbf{t},\tau)}\right|^{a}}{\left|\mathscr{B}(\mathbf{y},\mathbf{t},\tau)\right|^{b+1}}\leq \frac{C}{\|\mathbf{g}\|^{2b-2a+1}},
\end{aligned}\end{equation*}
 by \eqref{Y-Pm-a-b-1}-\eqref{Mj}.
Therefore,
\begin{equation*}
 \left|Y_j P_m(\mathbf{g})\right| \leq\frac{C}{\|\mathbf{g}\|^{Q+1}}, \quad \mathbf{g} \in \mathcal{N} \backslash\{\mathbf{0}\}.
\end{equation*}

Note that $Y_j=\frac{\partial}{\partial y_j}$ for $\mathbf{y} = \mathbf{0}$.
It follows from  Theorem \ref{th:pm-ex} that
 \begin{equation*} \begin{aligned}
  Y_jP_{m}(\mathbf{g}) = \sum_{j=0}^{r}C_{m,j}\int_{L_\varepsilon}\int_{S^{r-2}} (1-z^2)^{\frac{r-2}{2}} I_1 d\omega^{\prime} dz,
  \end{aligned}\end{equation*}
  with
  \begin{equation*} \begin{aligned}
 I_1&:=\frac{\partial}{\partial y_j}\left[\frac{ (\det\mathcal{B}^{O_{\mathbf{t}}\chi_{\omega^{\prime}}(z)})^{\frac{1}{2}}\left(\langle\mathcal{B}^{O_{\mathbf{t}}\chi_{\omega^{\prime}}(z)}
  \mathbf{y},\mathbf{y}\rangle+i|\mathbf{t}|z\right)^{m-j}}{\left(\langle\mathcal{B}^{O_{\mathbf{t}}\chi_{\omega^{\prime}}(z)}\mathbf{y},\mathbf{y}\rangle-i |\mathbf{t}|z\right)^{m+n+r-j}}\right].
   \end{aligned}\end{equation*}
  Note that
 \begin{equation*}\begin{aligned}
I_1|_{\mathbf{y} =  \mathbf{0}}=0.
\end{aligned}\end{equation*}
Therefore,
\begin{equation}\label{yjPm}
 \left|Y_j P_m(\mathbf{g})\right| \leq\frac{C}{\|\mathbf{g}\|^{Q+1}}, \quad  \mathbf{g} \in \mathcal{N} \backslash\{\mathbf{0}\}.
\end{equation}

It is known that the Carnot-Carath\'{e}odory metric $d_{c c}$ defined by \eqref{def:dcc} is left-invariant, and it is equivalent to the homogeneous metric $\rho$ in the sense that there exist $C_d, \widetilde{C}_d>0$ such that, for any $\mathbf{g}, \mathbf{h} \in \mathcal{N}$ (see \cite[(1.21)]{Ivanov}),
\begin{equation}\label{tri-dcc}
C_d \rho(\mathbf{g}, \mathbf{h}) \leq d_{c c}(\mathbf{g}, \mathbf{h}) \leq \widetilde{C}_d \rho(\mathbf{g}, \mathbf{h}) .
\end{equation}

 For any $\mathbf{g}, \mathbf{h}, \mathbf{g}_0$ with $\rho\left(\mathbf{g}_0, \mathbf{h}\right) \geq c \rho\left(\mathbf{g}, \mathbf{g}_0\right)$, we set $\mathbf{q}=\mathbf{h}^{-1} \cdot \mathbf{g}_0$, then by Lemma \ref{lem:mean value} and \eqref{tri-dcc}, we have
\begin{equation}\label{mean-value}
\begin{aligned}
 \left|P_m\left(\mathbf{h}^{-1} \cdot\mathbf{g}\right)-P_m\left(\mathbf{h}^{-1} \cdot \mathbf{g}_0\right)\right|
& =\left|P_m\left(\mathbf{q} \cdot\left(\mathbf{g}_0^{-1} \cdot \mathbf{g}\right)\right)-P_m(\mathbf{q})\right| \\
& \leq C d_{c c}\left(\mathbf{g}, \mathbf{g}_0\right) \max _{\substack{1 \leq j \leq 2n \\
\mathbf{u}: d_{c c}(\mathbf{u}, \mathbf{0}) \leq d_{c c}\left(\mathbf{g}, \mathbf{g}_0\right)}}\left|Y_j P_m(\mathbf{q} \cdot \mathbf{u})\right| \\
& \leq C \widetilde{C}_d \rho\left(\mathbf{g}, \mathbf{g}_0\right) \max _{\substack{1 \leq j \leq 2n \\
\mathbf{u}: d_{c c}(\mathbf{u}, \mathbf{0}) \leq d_{c c}\left(\mathbf{g},\mathbf{g}_0\right)}}\left|Y_j P_m(\mathbf{q} \cdot \mathbf{u})\right|.
\end{aligned}
\end{equation}

Since $\rho$ is a quasi-distance, by \eqref{tri-inequ}, \eqref{tri-dcc} and the fact that $\rho\left(\mathbf{g}_0, \mathbf{h}\right) \geq$ $c \rho\left(\mathbf{g}, \mathbf{g}_0\right)$, we have
$$
\begin{aligned}
\rho\left(\mathbf{u}, \mathbf{q}^{-1}\right) & \geq \frac{1}{C_\rho} \rho\left(\mathbf{0}, \mathbf{q}^{-1}\right)-\rho(\mathbf{0}, \mathbf{u}) \geq \frac{1}{C_\rho} \rho\left(\mathbf{g}_0, \mathbf{h}\right)-\frac{\widetilde{C}_d}{C_d} \rho\left(\mathbf{g}, \mathbf{g}_0\right) \\
& \geq \frac{1}{C_\rho} \rho\left(\mathbf{g}_0, \mathbf{h}\right)-\frac{\widetilde{C}_d}{c C_d} \rho\left(\mathbf{g}_0, \mathbf{h}\right) \\
& =\left(\frac{1}{C_\rho}-\frac{\widetilde{C}_d}{c C_d}\right) \rho\left(\mathbf{g}_0, \mathbf{h}\right) .
\end{aligned}
$$
Consequently,
$$
\left|Y_{j} P_m(\mathbf{q} \cdot\mathbf{ u})\right| \leq \frac{C}{\|\mathbf{q}\cdot \mathbf{u}\|^{Q+1}}=\frac{C}{\rho(\mathbf{u},\mathbf{q}^{-1})^{Q+1}} \lesssim \frac{1}{\rho\left(\mathbf{g}_0, \mathbf{h}\right)^{Q+1}},
$$
by using \eqref{yjPm}. Together with \eqref{mean-value}, we have
$$
\left|P_m( \mathbf{h}^{-1}\cdot\mathbf{g})-P_m\left(\mathbf{h}^{-1}\cdot\mathbf{g}_0\right)\right|\lesssim \frac{\rho\left(\mathbf{g}, \mathbf{g}_0\right)}{\rho\left(\mathbf{g}_0, \mathbf{h}\right)^{Q+1}}.
$$

  As a consequence, we see that the spectral  operator $\mathbb{P}_m$ is a standard Calder\'{o}n-Zygmund operator. Since $\mathbb{P}_m 1=0$ by \eqref{mu-k=0}, $\mathbb{P}_m$ is $L^2$ bounded by the well known $T 1$ Theorem. Therefore, $\mathbb{P}_m$ is  $L^p$ bounded, $1<p<\infty$.
To prove the last statement, note that
  \begin{equation*}\label{pmphi}
 \begin{aligned}
(\mathbb{P}_m\phi)(\mathbf{y},\mathbf{t})&=\langle\mathscr{P}_m, \phi_{(\mathbf{y},\mathbf{t})} \rangle
=\frac{1}{(2\pi)^r}\int_{\Bbb R^{2n+r}} Q_m(\mathbf{y^{\prime}},\tau) \widetilde{\phi}(\mathbf{y}-\mathbf{y^{\prime}}, \tau)e^{i\mathbf{t}\cdot \tau-2iB(\mathbf{y},\mathbf{y^{\prime}}) \cdot \tau}d\mathbf{y^{\prime}}d\tau\\
&=\frac{1}{(2\pi)^r}\int_{\Bbb R^{r}} \bigg[\widetilde{\phi}(\cdot, \tau)\ast_{\tau}Q_m(\cdot,\tau)\bigg](\mathbf{y})e^{i\mathbf{t}\cdot \tau}d\tau,
\end{aligned}\end{equation*}
by using Fubini's Theorem and \eqref{con-pm-2}, where we  use \eqref{twisor-con} in the last identity. The integral is convergent since $|Q_m(\cdot,\cdot) \widetilde{\phi}(\cdot, \cdot)|$ is in $L^1(\Bbb R^{2n+r})$ by estimate \eqref{esti-I1}. Since $\mathbb{P}_m$ is $L^2$ bounded, $\mathbb{P}_m\phi \in L^2(\Bbb R^{2n+r})$. Then for almost all $\mathbf{y} \in \Bbb R^{2n}$, $\mathbb{P}_m\phi(\mathbf{y}, \cdot) \in L^2(\Bbb R^{r})$. Consequently,
\begin{equation}\label{pm-phi-tilde}
 \widetilde{(\mathbb{P}_m\phi)}(\mathbf{y},\tau)= \left(\widetilde{\phi}(\cdot, \tau)\ast_{\tau}Q_m(\cdot,\tau)\right)(\mathbf{y}),
\end{equation}
by the Fourier inverse formula for $L^2$ functions on $\Bbb R^r$.

For $\phi\in L^2(\mathcal{N})$, there exists a sequence $\phi_n \in \mathcal{S}(\mathcal N)$ such that $\lim_{n\rightarrow \infty} \phi_n =\phi$ in the $L^2$ sense. We can define
$\mathbb{P}_m: L^2(\mathcal{N})\longmapsto  L^2(\mathcal{N})$
by letting
\begin{equation*}
\mathbb{P}_m\phi:=\lim_{n\rightarrow \infty}\mathbb{P}_m\phi_n,
\end{equation*}
in $L^2$, which is obviously independent of the choice of the sequence $\phi_n$, and consequently,
\begin{equation*}
\widetilde{\mathbb{P}_m\phi}=\lim_{n\rightarrow \infty}\widetilde{\mathbb{P}_m\phi_n},
\end{equation*}
in $L^2(\Bbb R^{2n+r})$ by Plancherel identity. Since
\begin{equation*}
\left\|\bigg(\widetilde{\phi}_n(\cdot, \tau)-\widetilde{\phi}(\cdot, \tau)\bigg)\ast_{\tau}Q_m(\cdot,\tau)\right\|_{L^2}\leq \left\|Q_m(\cdot,\tau)\right\|_{L^1}\left\|\widetilde{\phi}_n(\cdot, \tau)-\widetilde{\phi}(\cdot, \tau)\right\|_{L^2},
\end{equation*}
by Proposition \ref{cor:Minkowski}, where $\left\|Q_m(\cdot,\tau)\right\|_{L^1}$ is uniformly bounded by Lemma \ref{lem:L1,L2}, we see that
\begin{equation*}
 \widetilde{(\mathbb{P}_m\phi_n)}(\mathbf{y},\tau) =\bigg(\widetilde{\phi}_n(\cdot, \tau)\ast_{\tau}Q_m(\cdot,\tau)\bigg)(\mathbf{y})\longrightarrow \left(\widetilde{\phi}(\cdot, \tau)\ast_{\tau}Q_m(\cdot,\tau)\right)(\mathbf{y}),
\end{equation*}
in $L^2(\Bbb R^{2n+r})$. Hence, \eqref{pm-phi-tilde} holds for $\phi\in L^2(\mathcal{N}).$

Now we can use \eqref{pm-phi-tilde} repeatedly  to get
\begin{equation*}\begin{aligned}
\widetilde{\mathbb{P}_m(\mathbb{P}_m \phi)}(\mathbf{y},\tau)
=&\bigg[\bigg(\widetilde{\phi}(\cdot,\tau) \ast_\tau Q_{m}(\cdot,\tau)\bigg)\ast_{\tau}Q_{m}(\cdot,\tau)\bigg](\mathbf{y})\\
=&\left(\widetilde{\phi}(\cdot,\tau) \ast_\tau Q_{m}(\cdot,\tau)\right)(\mathbf{y})=\widetilde{(\mathbb{P}_m \phi)}(\mathbf{y},\tau),
\end{aligned}\end{equation*}
since
\begin{equation*}\label{eq:product}
\begin{aligned}
Q_{m_1}\ast_{\tau}Q_{m_2}(\mathbf{y},\tau)=&\sum_{|\mathbf{k}|=m_1}\sum_{|\mathbf{q}|=m_2}\widetilde{\mathscr  L}_{\mathbf{k} }^{(\mathbf{0} )}\ast_\tau\widetilde{\mathscr  L}_{\mathbf{q} }^{(\mathbf{0})}(\mathbf{y},\tau)\\
      =&\sum_{|\mathbf{k}|=m_1}\sum_{|\mathbf{q}|=m_2}\delta_{\mathbf{k}}^{(\mathbf{q})}\widetilde{\mathscr  L}_{\mathbf{k} }^{(\mathbf{0})}(\mathbf{y},\tau)\\
      =&\delta_{m_1}^{m_2}\sum_{|\mathbf{k}|=m_1}\widetilde{\mathscr  L}_{\mathbf{k} }^{(\mathbf{0})}(\mathbf{y},\tau)
      =\delta_{m_1}^{m_2}Q_{m_1}(\mathbf{y},\tau),
      \end{aligned}
         \end{equation*}
 by Proposition \ref{Proposition 1.2 in 7}. Hence,
\begin{equation*}
  \mathbb{P}_{m_1} \circ \mathbb{P}_{m_2} =\delta_{m_1}^{m_2}\mathbb{P}_{m_1},
\end{equation*}
and so $\mathbb{P}_m$ is a projection operator and mutually orthogonal in $L^2$. The theorem is proved. \qed

\section{The spectral projection decomposition}\label{section5}
 Strichartz \cite{S}  investigated the joint spectral
theory of the sub-Laplacian operators $\Delta_b$ and $iT$ on the Heisenberg group $\mathcal{H}_n$, where  $T=\frac{\partial}{\partial t}$.
The associated spectral projection  operators are convolution operators with kernel $K$ that can be described explicitly. These kernels are homogeneous functions of degree $-2n-2$ on $\mathcal{H}_{n}\setminus \{\mathbf{0}\}$ and are Calder\'{o}n-Zygmund kernels. Strichartz showed that the convolution operator is $L^p$ bounded by introducing another principal value associated the continuous function $ K$  on $\mathcal{N} \setminus \{\mathbf{0}\}$ defined by
\begin{equation}\label{p.v.f*k-2}
\langle \widehat{p.v.} K, \phi\rangle:=\lim_{\varepsilon\rightarrow 0^+}\int_{|\mathbf{t}|>\varepsilon}\int_{\Bbb R^{2n}}\phi
(\mathbf{y},\mathbf{t})K(\mathbf{y},\mathbf{t})d\mathbf{y}d\mathbf{t}.
\end{equation}
In order to estimating the boundedness of the operator associated with the Abel sum $\sum_{m=0}^{\infty}R^m\mathbb{P}_{m}$,
 we need  the explicit bound of each operator $\mathbb{P}_m$. Hence, here we need to choose the principal value defined by \eqref{p.v.f*k-2}  which help us to get the explicit bound of each operator $\mathbb{P}_m$.

\begin{pro}\label{expre-dis-2}
Let $\mathcal{N}$ be a nilpotent Lie group of step two and  $K$ be a continuous function on $\mathcal{N} \setminus \{\mathbf{0}\}$ which is homogeneous of degree $-Q$ and satisfies
  \begin{equation}\label{mu-k-2}
\mu_{K}:=\int_{S^{r-1}}\int_{\Bbb R^{2n}}K(\mathbf{y},\mathbf{t})d\mathbf{y}d\mathbf{t}=0.
\end{equation}
  Then the formula \eqref{p.v.f*k-2}
for any $\phi\in \mathcal{S}(\mathcal{N})$, defines a tempered distribution $\widehat{p.v.} K$ which is homogeneous of degree $-Q$. Conversely, suppose $\mathscr{K}$ is a tempered distribution which is homogeneous of degree $-Q$ and whose restriction on $\mathcal{N} \setminus \{\mathbf{0}\}$ is a continuous function $K$. Then \eqref{mu-k-2} holds and
\begin{equation*}
 \mathscr{K} = \widehat{p.v.} K+C\delta_\mathbf{0},
\end{equation*}
for some constant $C$.
\end{pro}
\begin{proof}
 Note that
  \begin{equation}\begin{aligned}\label{I1+I2}
    &\int_{|\mathbf{t}|>\varepsilon}\int_{\Bbb R^{2n}}K(\mathbf{y},\mathbf{t})\phi(\mathbf{y},\mathbf{t})d\mathbf{y}d\mathbf{t}\\
    =&\int_{U_{1,\varepsilon}}K(\mathbf{y},\mathbf{t})\bigg(\phi(\mathbf{y},\mathbf{t})-\phi(\mathbf{0})\bigg)d\mathbf{y}d\mathbf{t}
    +\int_{U_{2,\varepsilon}}K(\mathbf{y},\mathbf{t})\bigg(\phi(\mathbf{y},\mathbf{t})-\phi(\mathbf{0})\bigg)d\mathbf{y}d\mathbf{t}\\
    &\qquad+\int_{|\mathbf{t}|\geq 1}\int_{\Bbb R^{2n}}K(\mathbf{y},\mathbf{t})\phi(\mathbf{y},\mathbf{t})d\mathbf{y}d\mathbf{t}\\
    :=&I_{1,\varepsilon}+I_{2,\varepsilon}+I_2,
  \end{aligned}\end{equation}
by \eqref{mu-k-2}, where $U_{1, \varepsilon}:=\{(\mathbf{y},\mathbf{t}): |\mathbf{y}|\leq 1, \epsilon<|\mathbf{t}|\leq 1\}$, $U_{2, \varepsilon}:=\{(\mathbf{y},\mathbf{t}): |\mathbf{y}|\geq 1, \varepsilon <|\mathbf{t}|\leq 1\}$ with $\epsilon>0$. Since  $K$ is a continuous function on $\mathcal{N} \setminus \{\mathbf{0}\}$ which is homogeneous of degree $-Q$, then there exists a constant $C$ such that
 \begin{equation*}
   |K(\mathbf{y},\mathbf{t})|\leq \frac{C}{(|\mathbf{y}|^4+|\mathbf{t}|^2)^{\frac{Q}{4}}}.
 \end{equation*}
 Since $\phi$ is a Schwartz function on $\Bbb R^{2n+r}$, then $I_2$ is convergent. Similarly, $I_{2,\epsilon}$ is convergent. For Schwartz function $\phi$, there exists constants $C$ such that
\begin{equation*}
  |\phi(\mathbf{y},\mathbf{t})-\phi(\mathbf{0})|\leq C(|\mathbf{y}|+|\mathbf{t}|),
\end{equation*}
then
\begin{equation*}\label{esti-2}
\left | K(\mathbf{y},\mathbf{t})\big(\phi(\mathbf{y},\mathbf{t})-\phi(\mathbf{0})\big)\right|\leq \frac{C(|\mathbf{y}|+|\mathbf{t}|)}{(|\mathbf{y}|^4+|\mathbf{t}|^2)^{\frac{Q}{4}}},
\end{equation*}
which is integrable on $U_{1, \varepsilon}$, i.e.,  $I_{1,\epsilon}$ is convergent. Take limit on both sides of \eqref{I1+I2} to get
\begin{equation}\begin{aligned}\label{def:p.v.k}
  \langle \widehat{p.v.} K, \phi\rangle = &\lim_{\epsilon\rightarrow 0}\int_{|\mathbf{t}|>\varepsilon}\int_{\Bbb R^{2n}}K(\mathbf{y},\mathbf{t})\phi(\mathbf{y},\mathbf{t})d\mathbf{y}d\mathbf{t}\\
    =&\int_{|\mathbf{t}|<1}\int_{\Bbb R^{2n}}K(\mathbf{y},\mathbf{t})\bigg(\phi(\mathbf{y},\mathbf{t})-\phi(\mathbf{0})\bigg)d\mathbf{y}d\mathbf{t}+\int_{|\mathbf{t}|\geq 1}\int_{\Bbb R^{2n}}K(\mathbf{y},\mathbf{t})\phi(\mathbf{y},\mathbf{t})d\mathbf{y}d\mathbf{t}.
  \end{aligned}\end{equation}
 Moreover,
  \begin{equation*}\begin{aligned}
\langle \widehat{p.v.} K, \phi\circ \delta_\lambda\rangle
=&\lim_{\epsilon\rightarrow 0}\int_{|\mathbf{t}|\geq \varepsilon}\int_{\Bbb R^{2n}}K(\mathbf{y},\mathbf{t})\phi(\lambda\mathbf{y},\lambda^2\mathbf{t})d\mathbf{y}d\mathbf{t}\\
=&\lim_{\epsilon\rightarrow 0}\int_{|\widetilde{\mathbf{t}}|\geq \lambda^2 \varepsilon}\int_{\Bbb R^{2n}}K(\widetilde{\mathbf{y}},\widetilde{\mathbf{t}})\phi(\widetilde{\mathbf{y}},\widetilde{\mathbf{t}})d\widetilde{\mathbf{y}}d\widetilde{\mathbf{t}}
=\langle \widehat{p.v.}K, \phi\rangle,
\end{aligned}\end{equation*}
for any $\lambda>0$ by taking coordinates transformation $\widetilde{\mathbf{y}}=\lambda\mathbf{y},\widetilde{\mathbf{t}}=\lambda^2\mathbf{t}$ and using  the fact that $K$ is homogeneous of degree $-Q$. So that $\widehat{p.v.} K$ is homogeneous of degree $-Q$ as a distribution by \eqref{homo-2}. \par
Conversely, if $\mathscr{K}$ is a homogeneous distribution which agrees with $K$ away from $\mathbf{0}$, then $\mathscr{K}-\widehat{p.v.}K$ is supported at $\mathbf{0}$ since $\widehat{p.v.}K$  agrees with $K$ on $\mathcal{N}\backslash \{\mathbf{0}\}$ by definition \eqref{p.v.f*k-2}. Hence it is a linear combination of $\delta_{\mathbf{0}}$ and its derivatives. Let $\mathbf{x}=(\mathbf{y},\mathbf{t})$. But $\bigg(\frac{\partial}{\partial \mathbf{x}}\bigg)^{\mathbf{I}}\delta_{\mathbf{0}} $ is homogeneous of degree $-Q-d(\mathbf{I})$, where $d(\mathbf{I}):=\sum_{j=1}^{2n}I_j+2\sum_{\alpha=1}^{r}I_{2n+\alpha}$ with $\mathbf{I}=(I_1,\ldots,I_{2n+r})$. So by homogeneity we must have $\mathscr{K}-\widehat{p.v.}K=C\delta_{\mathbf{0}}$.

Note that for $\lambda>0$,
\begin{equation}\begin{aligned}\label{homo-r}
\int_{S^{r-1}}\int_{\Bbb R^{2n}}K(\mathbf{y},\lambda\dot{\mathbf{t}})d\mathbf{y}d\mathbf{\dot{t}}
=&\int_{S^{r-1}}\int_{\Bbb R^{2n}}K(\sqrt{\lambda}\frac{\mathbf{y}}{\sqrt{\lambda}},\lambda\dot{\mathbf{t}})d\mathbf{y}d\mathbf{\dot{t}}\\
=&\lambda^{-n-r}\int_{S^{r-1}}\int_{\Bbb R^{2n}}K(\frac{\mathbf{y}}{\sqrt{\lambda}},\dot{\mathbf{t}})d\mathbf{y}d\mathbf{\dot{t}}\\
=&\lambda^{-r}\int_{S^{r-1}}\int_{\Bbb R^{2n}}K(\widetilde{\mathbf{y}},\dot{\mathbf{t}})d\widetilde{\mathbf{y}}d\mathbf{\dot{t}},
\end{aligned}\end{equation}
where the second identity holds since $K$ is homogeneous of degree $-Q$ and we take a coordinate transformation $\widetilde{\mathbf{y}}=\frac{\mathbf{y}}{\sqrt{\lambda}}$ in the third identity.

 Consider the distribution $\mathscr{F}$ defined by
\begin{equation*}\begin{aligned}
   \langle\mathscr{F} , \phi\rangle= &\int_{|\mathbf{t}|<1}\int_{\Bbb R^{2n}}K(\mathbf{y},\mathbf{t})[\phi(\mathbf{y},\mathbf{t})-\phi(\mathbf{0})]d\mathbf{y}d\mathbf{t}
    +\int_{|\mathbf{t}|\geq 1}\int_{\Bbb R^{2n}}K(\mathbf{y},\mathbf{t})\phi(\mathbf{y},\mathbf{t})d\mathbf{y}d\mathbf{t},
  \end{aligned}\end{equation*}
for $\phi \in \mathcal{S}(\mathcal{N})$. $\mathscr{F}$ agrees with $\mathscr{K}$ away from $\mathbf{0}$, so again $\mathscr{F}-\mathscr{K}=\sum_{I}a_{I}\bigg(\frac{\partial}{\partial \mathbf{x}}\bigg)^{\mathbf{I}}\delta_{\mathbf{0}}$. Since $\mathscr{K}$ is homogeneous of degree $-Q$ for any $\phi \in \mathcal{S}(\mathcal{N})$ and $\lambda>0$, then it follows from \eqref{homo-2} that
\begin{equation*}\begin{aligned}
\langle \mathscr{F}, \phi\circ \delta_\lambda\rangle-\langle \mathscr{F}, \phi\rangle=\langle \mathscr{F}-\mathscr{K}, \phi\circ \delta_\lambda\rangle-\langle \mathscr{F}-\mathscr{K}, \phi\rangle= \sum_{\mathbf{I}}a_{\mathbf{I}}(\lambda^{d(\mathbf{I})}-1)\bigg(\frac{\partial \phi}{\partial \mathbf{x}}\bigg)^{\mathbf{I}}(\mathbf{0}),
\end{aligned}\end{equation*}
 which is bounded as $\lambda\rightarrow 0$. Note that  for $\lambda<1$,
\begin{equation*}\begin{aligned}
\langle \mathscr{F}, \phi\circ \delta_\lambda\rangle-\langle \mathscr{F}, \phi\rangle
=&\int_{|\mathbf{t}|<\lambda^2}\int_{\Bbb R^{2n}}[\phi(\mathbf{y},\mathbf{t})-\phi(\mathbf{0})]
K(\mathbf{y},\mathbf{t})d\mathbf{y}d\mathbf{t}+\int_{|\mathbf{t}|\geq \lambda^2}\int_{\Bbb R^{2n}}K(\mathbf{y},\mathbf{t})\phi(\mathbf{y},\mathbf{t})d\mathbf{y}d\mathbf{t}\\
-&\int_{|\mathbf{t}|<1}\int_{\Bbb R^{2n}}[\phi(\mathbf{y},\mathbf{t})-\phi(\mathbf{0})]
K(\mathbf{y},\mathbf{t})d\mathbf{y}d\mathbf{t}-\int_{|\mathbf{t}|\geq 1}\int_{\Bbb R^{2n}}K(\mathbf{y},\mathbf{t})\phi(\mathbf{y},\mathbf{t})d\mathbf{y}d\mathbf{t}\\
=&\phi(\mathbf{0})\int_{\lambda^2\leq|\mathbf{t}|<1}\int_{\Bbb R^{2n}}K(\mathbf{y},\mathbf{t})d\mathbf{y}d\mathbf{t}\\
=&\phi(\mathbf{0})\int_{\lambda^2}^{1}|\mathbf{t}|^{r-1} d|\mathbf{t}|\int_{S^{r-1}}\int_{\Bbb R^{2n}}K(\mathbf{y},|\mathbf{t}|\dot{\mathbf{t}})d\mathbf{y}d\mathbf{\dot{t}}\\
=&\phi(\mathbf{0})\mu_{K}\int_{\lambda^2}^{1}|\mathbf{t}|^{-1} d|\mathbf{t}|\\
=&-\phi(\mathbf{0})\mu_{K}\log \lambda^2 \longrightarrow \infty,\,\,\,\,\,\,\text{as} \, \lambda\rightarrow 0,
\end{aligned}\end{equation*}
where the forth identity holds by using \eqref{homo-r}.
If $\mu_{K}\neq 0$, then there is an contraction when $\phi(\mathbf{0})\neq 0$. Hence, \eqref{mu-k-2} holds. The proposition is proved.
\end{proof}

Since $\mathscr{P}_m$ is  a tempered distribution which is homogeneous of degree $-Q$ and whose restriction to $\mathcal{N} \setminus \{\mathbf{0}\}$ is a continuous function $P_m$, then by Proposition \ref{expre-dis-2},
\begin{equation}\label{def-con-pm-2}
\begin{aligned}
(\phi\ast\mathscr{P}_m)(\mathbf{y},\mathbf{t})=&\lim_{\epsilon\rightarrow 0}\int_{|\mathbf{t}^{\prime}|>\epsilon}\int_{\Bbb R^{2n}}P_m(\mathbf{y^{\prime}},\mathbf{t^{\prime}})\phi_{(\mathbf{y},\mathbf{t})}(\mathbf{y^{\prime}},\mathbf{t^{\prime}})d\mathbf{y^{\prime}}d\mathbf{t^{\prime}}
+C\phi(\mathbf{y},\mathbf{t})\\
:=&\phi\ast\widehat{p.v.}P_m(\mathbf{y},\mathbf{t})+C\phi(\mathbf{y},\mathbf{t}),
\end{aligned}
\end{equation}
for some constant $C$.

In \cite{S}, Strichartz gave the following boundedness of the  principal value of integral $\phi\ast \widehat{p.v.} K$ associated with the continuous function $ K$ defined by \eqref{p.v.f*k-2}.
\begin{lem}\cite[Lemma 3.1]{S}
  Let $K(\mathbf{y},t)$ be an odd function homogeneous of degree $-2n-2$ on the Heisenberg group $G$. Then the operator norm of $\phi\ast \widehat{p.v.} K$ on $L^p$ is bounded by $c_p\int_{\Bbb R^{2n}}|K(\mathbf{y},1)|d\mathbf{y}$, where $c_p$ is a constant depending on $p$.
\end{lem}
We generalized this lemma to the case of nilpotent  Lie group $\mathcal{N}$ of step two.
\begin{lem}\label{lem:esti}
Let $K(\mathbf{y}, \mathbf{t})$ be  homogeneous of degree $-2 n-2r$ on  $\mathcal{N}$ and satisfy \eqref{mu-k-2}. Then the operator norm of $\phi\ast \widehat{p.v.}K$ on $L^{p}$ is bounded by $c_{p} \int_{S^{r-1}}\int_{\mathbb{R}^{2n}}|K(\mathbf{y}, \mathbf{t})| d \mathbf{y}d\mathbf{t},$  where $c_p$ is a constant depending on $p$, $1<p<\infty$.
\end{lem}
\begin{proof} Since $K(\mathbf{y}, \mathbf{t})$ be  homogeneous of degree $-2 n-2r$ on  $\mathcal{N}$ and satisfy \eqref{mu-k-2}, then \eqref{def:p.v.k} is well defined and
\begin{equation*}
\begin{aligned}
\phi\ast\widehat{ p.v. } K(\mathbf{y}, \mathbf{t})
&=\lim_{\varepsilon\rightarrow 0^+}\int_{|s|>\varepsilon^2}\int_{\Bbb R^{2n}}\phi\bigg(\mathbf{y}-\mathbf{w},\mathbf{t}-\mathbf{s}-2B(\mathbf{y},\mathbf{w})\bigg)\cdot|\mathbf{s}|^{-n-r}\cdot K(\frac{\mathbf{w}}{\sqrt{|\mathbf{s}|}},\dot{\mathbf{s}})d\mathbf{w}d\mathbf{s}\\
&=\lim_{\varepsilon\rightarrow 0^+}\int_{|s|>\varepsilon^2}\int_{\Bbb R^{2n}}\phi\bigg(\mathbf{y}-\sqrt{|\mathbf{s}|}\widetilde{\mathbf{w}},\mathbf{t}-\mathbf{s}
-2\sqrt{|\mathbf{s}|}B(\mathbf{y},\widetilde{\mathbf{w}})\bigg)\cdot|\mathbf{s}|^{-r}\cdot K(\widetilde{\mathbf{w}},\dot{\mathbf{s}})d\mathbf{\widetilde{w}}d\mathbf{s}\\
&=\lim_{\varepsilon\rightarrow 0^+}\int_{\Bbb R^{2n}} \int_{S^{r-1}}\int_{\varepsilon^2}^{+\infty}\phi\bigg(\mathbf{y}-\sqrt{|\mathbf{s}|}\widetilde{\mathbf{w}},\mathbf{t}-|\mathbf{s}|\dot{\mathbf{s}}
-2\sqrt{|\mathbf{s}|}B(\mathbf{y},\widetilde{\mathbf{w}})\bigg) K(\widetilde{\mathbf{w}},\dot{\mathbf{s}})d\widetilde{\mathbf{w}}\frac{d|\mathbf{s}|}{|\mathbf{s}|} d\dot{\mathbf{s}}\\
&=\lim_{\varepsilon\rightarrow 0^+}2\int_{\Bbb R^{2n}} \int_{S^{r-1}}\int_{\varepsilon}^{+\infty}\phi\bigg(\mathbf{y}-u \mathbf{w}, t-u^{2}\dot{\mathbf{s}}-2uB(\mathbf{y},\mathbf{w})\bigg) \frac{d u}{u} K(\mathbf{w}, \dot{s})d \mathbf{w}d\dot{s}\\
&:=2\int_{\Bbb R^{2n}} \int_{S^{r-1}}H\phi_{\mathbf{w},\dot{\mathbf{s}}}(\mathbf{y},\mathbf{t}) K(\mathbf{w}, \dot{s})d \mathbf{w}d\dot{s},
\end{aligned}
\end{equation*}
where
\begin{equation*}
H\phi_{\mathbf{w},\dot{\mathbf{s}}}(\mathbf{y},\mathbf{t}):=\lim_{\varepsilon\rightarrow 0^+}\int_{\varepsilon}^{+\infty}\phi\bigg(\mathbf{y}-u \mathbf{w}, \mathbf{t}-u^{2}\dot{\mathbf{s}}-2uB(\mathbf{y},\mathbf{w})\bigg) \frac{d u}{u},
\end{equation*}
for fixed $\mathbf{w}\in \Bbb R^{2n}, \dot{\mathbf{s}}\in S^{r-1}$.
Here the first identity holds since $K(\mathbf{y}, \mathbf{t})$ is homogeneous of degree $-2 n-2r$, we take the coordinate transform
$$\widetilde{\mathbf{w}}:=\frac{\mathbf{w}}{\sqrt{|\mathbf{s}|}},\,\,\,\,\,\, u:=\sqrt{|\mathbf{s}|},$$
in the second and forth identity respectively.
Then by  Minkowski's inequality, we have
\begin{equation*}
\begin{aligned}
 \| \phi \ast\widehat{p.v.} K\|_{L^p}
=&2\bigg(\int_{\Bbb R^{2n}\times \Bbb R^{r}}\bigg|\int_{\Bbb R^{2n}\times S^{r-1}}H\phi_{\mathbf{w},\dot{\mathbf{s}}}(\mathbf{y},\mathbf{t}) K(\mathbf{w}, \dot{s})d \mathbf{w}d\dot{s}\bigg|^p d\mathbf{y}d\mathbf{t}\bigg)^{\frac{1}{p}}\\
\leq& 2\int_{\Bbb R^{2n}\times S^{r-1}}\bigg(\int_{\Bbb R^{2n}\times \Bbb R^{r}}\bigg|H\phi_{\mathbf{w},\dot{\mathbf{s}}}(\mathbf{y},\mathbf{t}) K(\mathbf{w}, \dot{s})\bigg|^p d\mathbf{y}d\mathbf{t}\bigg)^{\frac{1}{p}}d \mathbf{w}d\dot{s}\\
=&2\int_{\Bbb R^{2n}\times S^{r-1}}\bigg(\int_{\Bbb R^{2n}\times \Bbb R^{r}}\bigg|H\phi_{\mathbf{w},\dot{\mathbf{s}}}(\mathbf{y},\mathbf{t})\bigg|^p d\mathbf{y}d\mathbf{t}\bigg)^{\frac{1}{p}} |K(\mathbf{w}, \dot{s})|d \mathbf{w}d\dot{s}.
\end{aligned}
\end{equation*}
Then  the lemma will follow if we can show that $H\phi_{\mathbf{w},\dot{\mathbf{s}}}$ is $L^p$ bounded and its bound is is independent of $\mathbf{w}$ and $\dot{\mathbf{s}}$.

 We can rotate $\mathbf{w}$ to be of the form $(\rho, 0, \ldots, 0)$, $\rho$ real and $\dot{s}:=(1,0,\ldots,0)$. Let $M_{\beta}:=2\sum_{k=1}^{2n}B^{\beta}_{k1}y_{k}$. Note that $M_{\beta}$ is independent of $y_1$ and $M_{\beta}\neq 0$ for $\mathbf{y}\neq\mathbf{0}$. Denote $y^{\prime\prime}:=(y_2,\ldots,y_{2n})$. Then
 \begin{equation*}
\begin{aligned}
\|H\phi_{\mathbf{w},\dot{\mathbf{s}}}\|_{L^p}^{p}
=&\int_{\Bbb R^{2n}\times \Bbb R^{r}}\bigg|\lim_{\varepsilon\rightarrow 0^+}\int_{\varepsilon}^{+\infty}\phi\bigg(y_{1}-u \rho, y^{\prime\prime}, t_1-u^{2}-u\rho M_1,t_2-u\rho M_2,\ldots\bigg) \frac{d u}{u}\bigg|^p d\mathbf{y}d\mathbf{t}\\
=&\int_{\Bbb R^{2n-1}} I(y^{\prime\prime})dy^{\prime\prime},
\end{aligned}
\end{equation*}
where
\begin{equation*}
\begin{aligned}
I(y^{\prime\prime})=\int_{\Bbb R\times \Bbb R^{r}}\bigg|\lim_{\varepsilon\rightarrow 0^+}\int_{\varepsilon}^{+\infty}\phi\bigg(y_{1}-u \rho, y^{\prime\prime}, t_1-u^{2}-u\rho M_1,t_2-u\rho M_2,\ldots\bigg) \frac{d u}{u}\bigg|^p dy_1d\mathbf{t}.
\end{aligned}
\end{equation*}
 Then we have
 \begin{equation}\label{Hf-1}
\begin{aligned}
 &I(y^{\prime\prime})\\
 =&\int_{\Bbb R\times \Bbb R^{r}}\bigg|\lim_{\varepsilon\rightarrow 0^+}\int_{\varepsilon}^{+\infty}\phi\bigg(\rho^2y^{\prime}_{1}-\rho^2 u^{\prime}, y^{\prime\prime}, \rho^2t_1^{\prime}-\rho^2u^{\prime2}-\rho^2 u^{\prime}M_1,\rho^2t_2^{\prime}-\rho^2u^{\prime}M_2,\ldots\bigg) \frac{d u^{\prime}}{u^{\prime}}\bigg|^p \rho^{2+2r}dy_{1}^{\prime}d\mathbf{t}^{\prime}\\
 =&\int_{\Bbb R\times \Bbb R^{r}}\bigg|\lim_{\varepsilon\rightarrow 0^+}\int_{\varepsilon}^{+\infty}\widetilde{\phi}\bigg(y^{\prime}_{1}- u^{\prime}, y^{\prime\prime}, t_1^{\prime}-u^{\prime2}- u^{\prime}M_1, t_2^{\prime}-u^{\prime}M_2,\ldots\bigg) \frac{d u^{\prime}}{u^{\prime}}\bigg|^p \rho^{2+2r}dy_{1}^{\prime}d\mathbf{t}^{\prime}\\
 =& \int_{\Bbb R\times\Bbb R^{r}}\bigg|\lim_{\varepsilon\rightarrow 0^+}\int_{\varepsilon}^{+\infty}\widetilde{\phi}\bigg(M_1\widetilde{y}_{1}- M_1\widetilde{u}, y^{\prime\prime}, M_1^2\widetilde{t}_1-M_1^2\widetilde{u}^{2}-M_1^2\widetilde{u},\\
  &\qquad\qquad\qquad\qquad\qquad M_1M_2\widetilde{t}_2-M_1M_2\widetilde{u},\ldots\bigg) \frac{d \widetilde{u}}{\widetilde{u}}\bigg|^p \rho^{2+2r}M_1^{1+r}M_2\cdots M_{r}d\widetilde{y}_{1}d\widetilde{\mathbf{t}}\\
 =&\int_{\Bbb R\times\Bbb R^{r}}\bigg|\lim_{\varepsilon\rightarrow 0^+}\int_{\varepsilon}^{+\infty}\widetilde{\widetilde{\phi}}\bigg(\widetilde{y}_{1}- \widetilde{u}, y^{\prime\prime}, \widetilde{t}_1-\widetilde{u}^{2}-\widetilde{u},\widetilde{t}_2-\widetilde{u},\ldots\bigg) \frac{d\widetilde{ u}}{\widetilde{u}}\bigg|^p \rho^{2+2r}M_1^{2+r}M_{2}\cdots M_{r}d\widetilde{y}_{1}d\widetilde{\mathbf{t}}\\
 =&\rho^{2+2r}M_1^{2+r}M_{2}\cdots M_{r}\int_{\Bbb R\times\Bbb R^{r}}\bigg|\mathcal{H} \widetilde{\widetilde{\phi}}_{y^{\prime\prime}}(\widetilde{y}_1,\widetilde{\mathbf{t}})\bigg|^pd\widetilde{y}_{1}d\widetilde{\mathbf{t}},
\end{aligned}
\end{equation}
where
\begin{equation}\label{integral-along curve}
\begin{aligned}
\mathcal{H} \widetilde{\widetilde{\phi}}_{y^{\prime\prime}}(\widetilde{y}_1,\widetilde{\mathbf{t}}):=&\lim_{\varepsilon\rightarrow 0^+}\int_{\varepsilon}^{+\infty}\widetilde{\widetilde{\phi}}\bigg(\widetilde{y}_{1}- \widetilde{u}, y^{\prime\prime}, \widetilde{t}_1-\widetilde{u}^{2}-\widetilde{u}, \widetilde{t}_2-\widetilde{u},\ldots\bigg) \frac{d \widetilde{u}}{\widetilde{u}}\\
=&p.v. \int_{0}^{+\infty}\widetilde{\widetilde{\phi}}_{y^{\prime\prime}}\bigg((\widetilde{y}_{1},  \widetilde{t}_1, \widetilde{t}_2,\ldots)-\Upsilon(\widetilde{u})\bigg)\frac{d \widetilde{u}}{\widetilde{u}}.
\end{aligned}
\end{equation}
Here we take coordinates $y_{1}=\rho^2y_{1}^{\prime}, u=\rho u^{\prime}, t_{\beta}=\rho^2 t^\prime_{\beta}$ in the first identity, $y_{1}^{\prime}=M_1\widetilde{y}_{1}, u^{\prime}=M_1 \widetilde{u}, t^{\prime}_{\beta}=M_1M_{\beta} \widetilde{t}_{\beta}$ in the third identity, and we take
\begin{equation}\label{phi-1}
\widetilde{\phi}(y_1,y^{\prime\prime},\mathbf{t}):=\phi(\rho^2 y_1, y^{\prime\prime},\rho^2 \mathbf{t})
\end{equation}
 and \begin{equation}\label{phi-2}
 \widetilde{\widetilde{\phi}}(y_1,y^{\prime\prime},t_1,\ldots, t_r):=\widetilde{\phi}(M_1y_1,y^{\prime\prime}, M_1^2t_1, M_1M_{2}t_2,\ldots,M_1M_{r}t_r )
 \end{equation} in the second and forth identity respectively. \par
 Note that
 \begin{equation*}
 \begin{aligned}
 \Bbb R&\longrightarrow \Bbb R^{r+1}\\
 \Upsilon(u)&:=(u,u^2+u,u,\ldots,u)
 \end{aligned}
 \end{equation*}
 is well-curved since
 $\Upsilon(0)=0$ and
 \begin{equation*}
 \Upsilon(u)=u\Upsilon^{\prime}(0)+\frac{u^2}{2}\Upsilon^{\prime\prime}(0),
 \end{equation*}
 where  $\Upsilon^{\prime}(0)=(1,1,1,\ldots,1)$ and $\Upsilon^{\prime\prime}(0)=(0,2,0,\ldots,0)$, that is,  a segment of the curve containing the origin lies in the subspace of $\Bbb R^{r+1}$ spanned by $\Upsilon^{\prime}(0)$ and $\Upsilon^{\prime\prime}(0)$. By \cite{Stein1},  the principal value integral  operators  defined by
$$
\mathcal{H} \phi(x)=p.v.\int_{-1}^1 \phi(x-\gamma(t)) \frac{d t}{t}
$$
for smooth function $f$ on $\Bbb R^{r+1}$, has the following property:
$$
\|\mathcal{H}\phi\|_p \leqslant A_p\|\phi\|_p, \quad 1<p<\infty,
$$
provided $\gamma$ is well-curved, where $A_p$ is a constant independent on $\phi$.
Then $\mathcal{H} \widetilde{\widetilde{\phi}}_{y^{\prime\prime}}(\widetilde{y}_1,\widetilde{\mathbf{t}})$ defined by \eqref{integral-along curve}
is a principal value integral along the well-curved curve $\Upsilon(u):=(u,u^2+u,u,\ldots,u)$. Then by theorem 1 in \cite{Stein1}, we have
 \eqref{integral-along curve} is $L^p$ bounded, that is
 $$
\|\mathcal{H}\widetilde{\widetilde{\phi}}_{y^{\prime\prime}}\|_p \leqslant A_p\|\widetilde{\widetilde{\phi}}_{y^{\prime\prime}}\|_p, \quad 1<p<\infty,
$$
and $A_p$ is independent of $\rho$.
  Hence,
 \begin{equation*}
\begin{aligned}
 &\|H\phi_{\mathbf{w},\dot{\mathbf{s}}}\|_{L^p}^{p}\\
 =&\int_{\Bbb R^{2n-1}}\rho^{2+2r}M_1^{2+r}M_{2}\cdots M_{r}\int_{\Bbb R\times\Bbb R^{r}}\bigg|\mathcal{H} \widetilde{\widetilde{\phi}}_{y^{\prime\prime}}(\widetilde{y}_1,\widetilde{\mathbf{t}})\bigg|^p
 d\widetilde{y}_{1}d\widetilde{\mathbf{t}}dy^{\prime\prime}\\
 \leq&\int_{\Bbb R^{2n-1}} A_{p}^{p}\|\widetilde{\widetilde{\phi}}_{y^{\prime\prime}}\|^p \rho^{2+2r}M_1^{1+r}M_{2}\cdots M_{r}dy^{\prime\prime}\\
  =&\int_{\Bbb R^{2n-1}} A_{p}^{p}\int_{\Bbb R\times\Bbb R^{r}}\bigg|\phi(\rho^2 M_1 y_1, y^{\prime\prime}, \rho^2M_1^2\mathbf{t}_1,\rho^2M_1M_2\mathbf{t}_2,\ldots)\bigg|^p dy_{1} d\mathbf{t}\rho^{2+2r}M_1^{2+r}M_{2}\cdots M_{r}dy^{\prime\prime}\\
  =&A_{p}^{p}\int_{\Bbb R^{2n}\times \Bbb R^{r}}\bigg|\phi(y_1, y^{\prime\prime}, \mathbf{t})\bigg|^p dy_{1}dy^{\prime\prime}d\mathbf{t}\\
  =&A_{p}^{p}\|\phi\|^{p},
\end{aligned}
\end{equation*}
where $A_p$ is independent of $\mathbf{w}$ and $\dot{\mathbf{s}}$. Here we use \eqref{Hf-1} in the first identity, \eqref{phi-1}-\eqref{phi-2} in the second identity. The proposition is proved.
\end{proof}

 \begin{pro}\label{pro:rmpm} On a non-degenerate nilpotent Lie group of step two $\mathcal{N}$,  $P_m(\mathbf{y}, \mathbf{t})$ defined by \eqref{Pm-Qm} satisfies
\begin{equation*}\begin{aligned}\label{RP-2}
&\sum_{m=0}^{\infty}R^mP_{m}(\mathbf{y}, \mathbf{t})
=\frac{(n+r-1)!2^{n-r}M^{r}}{ \pi^{n+r}(1+R)^n}
\int_{S^{r-1}}\frac{(\det\mathcal{B}^\tau)^{\frac{1}{2}}}{\bigg[\langle \mathcal{B}^\tau\mathbf{y},\mathbf{y}\rangle-iM\mathbf{t} \cdot \tau\bigg]^{n+r}}d\tau,
\end{aligned}\end{equation*}
for $\mathbf{y}\neq \mathbf{0}$.
\end{pro}
\begin{proof}
It follows from \eqref{eq:pm-r} and the fact
\begin{equation*}
\sum_{m=0}^{\infty}R^mL_m^{\alpha}(x)=\frac{1}{(1-R)^{\alpha+1}}e^{-\frac{Rx}{1-R}}, \qquad |R|<1, \,\,x>0
\end{equation*}
 (cf. \cite[P. 242]{special}) that
\begin{equation*}\begin{aligned}\label{RP-1}
\sum_{m=0}^{\infty}R^mP_{m}(\mathbf{y}, \mathbf{t})=&\frac{1}{4^r \pi^{n+r}}\int_{S^{r-1}}\frac{(\det\mathcal{B}^{\dot{\tau}})^{\frac{1}{2}}}{\sigma^{n+r}}d\dot{\tau}\int_{0}^{\infty} s^{n+r-1} e^{-\frac{\sigma-i\mathbf{t}\cdot\dot{\tau}}{2\sigma} s} \sum_{m=0}^{\infty}R^m L_{m}^{(n-1)}(s) d s\\
=&\frac{1}{4^r \pi^{n+r}(1-R)^{n}}\int_{S^{r-1}}\frac{(\det\mathcal{B}^{\dot{\tau}})^{\frac{1}{2}}}{\sigma^{n+r}}\int_{0}^{\infty} s^{n+r-1}e^{-(\frac{\sigma-i\mathbf{t}\cdot\dot{\tau}}{2\sigma}+\frac{R}{1-R})s}ds  d\dot{\tau}\\
=&\frac{(n+r-1)!}{4^r \pi^{n+r}(1-R)^{n}}\int_{S^{r-1}}\frac{(\det\mathcal{B}^{\dot{\tau}})^{\frac{1}{2}}}{\sigma^{n+r}}\frac{1}{(\frac{\sigma-i\mathbf{t}\cdot\dot{\tau}}{2\sigma}+\frac{R}{1-R})^{n+r}} d\dot{\tau}\\
=&\frac{(n+r-1)!2^{n-r}M^{r}}{ \pi^{n+r}(1+R)^n}\int_{S^{r-1}}\frac{(\det\mathcal{B}^{\dot{\tau}})^{\frac{1}{2}}}{\bigg(\sigma-iM\mathbf{t} \cdot \dot{\tau}\bigg)^{n+r}}d\dot{\tau},
\end{aligned}\end{equation*}
where $\sigma:=\langle \mathcal{B}^{\dot{\tau}}\mathbf{y},\mathbf{y}\rangle$, $M:=\frac{1-R}{R+1}$ and we take coordinates transformation $s=2|\tau|\sigma$ in the first identity.
 We use
\begin{equation*}
\int_{0}^{\infty} s^{k}e^{-\omega s}ds=\frac{k!}{\omega^{k+1}},
\end{equation*}
when ${\rm Re}\omega>0$ in the third identity.
\end{proof}

\begin{pro}\label{pro:esti-PR}
On a non-degenerate nilpotent Lie group of step two $\mathcal{N}$, the operator norm of $\widehat{p.v.}\sum_{m=0}^{\infty}R^mP_{m} $ on $L^{p}$ is bounded by $ \frac{c_{p}}{(R+1)^n}, 1<p<\infty$, $0<R<1$, where $c_p$ is a constant only depending on $p$.
\end{pro}
\begin{proof} Since $\sum_{m=0}^{+\infty} R^{m}P_m(\mathbf{y}, \mathbf{t})$ is also homogeneous of degree $-Q$ and satisfies  \eqref{mu-k-2} with $K=\sum_{m=0}^{+\infty} R^{m}P_m$,  then  the operator norm of $\widehat{p.v.}\sum_{m=0}^{\infty}R^mP_{m} $ on $L^{p}$ is bounded by $$ c_p\int_{\mathbb{R}^{2n}}\int_{S^{r-1}}\bigg|\sum_{m=0}^{+\infty} R^{m}P_m(\mathbf{y}, \mathbf{t})\bigg| d\mathbf{y}d\mathbf{t},$$  by Lemma \ref{lem:esti}.
Let  $M:=\frac{1-R}{R+1}$. Note that
\begin{equation*}
\begin{aligned}
&\int_{\mathbb{R}^{2n}}\int_{S^{r-1}}\bigg| \int_{S^{r-1}}\frac{(\det\mathcal{B}^\tau)^{\frac{1}{2}}}
{\bigg[\langle \mathcal{B}^\tau\mathbf{y},\mathbf{y}\rangle-iM\mathbf{t}\cdot \tau\bigg]^{n+r}}
d\tau\bigg|d\mathbf{y}d\mathbf{t}
\leq  \int_{S^{r-1}} \int_{S^{r-1}}(\det\mathcal{B}^\tau)^{\frac{1}{2}} I_1d\tau d\mathbf{t},
\end{aligned}
\end{equation*}
where
\begin{equation*}
\begin{aligned}
I_1:=\int_{\Bbb R^{2n}}\frac{1}
{\bigg|\langle \mathcal{B}^\tau\mathbf{y},\mathbf{y}\rangle-iM \mathbf{t} \cdot \tau\bigg|^{n+r}}
d\mathbf{y}.
\end{aligned}
\end{equation*}
By Lemma \ref{lem:B-y}, there exists a constant $C>0$ such that
\begin{equation*}
\begin{aligned}
I_1 \lesssim \int_{0}^{\infty}\frac{\rho^{2n-1}d\rho}{(C\rho^{4}+M^2)^{\frac{n+r}{2}}}
=C^\prime\int_{0}^{\infty}\frac{ u^{n-1}du}{(u^{2}+M^2)^{\frac{n+r}{2}}}=\frac{C^\prime}{M^r},
\end{aligned}
\end{equation*}
where we take the coordinate transform $\sqrt{C}\rho^{2}=u$ in the first identity. It follows from Proposition \ref{pro:rmpm} that
\begin{equation*}
\begin{aligned}
\int_{\mathbb{R}^{2n}}\int_{S^{r-1}}\bigg|\sum_{m=0}^{+\infty} R^{m}P_m(\mathbf{y}, \mathbf{t})\bigg| d\mathbf{y}d\mathbf{t}
\leq&\frac{C^\prime}{(R+1)^n}.
\end{aligned}
\end{equation*}
The proposition is proved.
\end{proof}

\begin{proof}[Proof of Theorem \ref{thm:f-decom}.]
In \cite[Theorem 4.1]{CMW}, for $\phi\in \mathcal{S}(\mathcal{N})$,  it is known  that
\begin{equation*}\label{eq:app}
\sum_{m=0}^{\infty}R^{m}\mathbb{P}_{m}\phi\rightarrow \phi\qquad     {\rm as }\quad R\rightarrow 1^-.
\end{equation*}
 By Theorem \ref{thm:con-Pm-bounded},
$ \mathbb{P}_{m}\phi \in L^p(\mathcal{N})$ for any $\phi \in L^p(\mathcal{N})$.
 On the other hand,  the  operator norm of convolution operator  $\widehat{p.v.} \sum_{m=0}^{\infty}R^mP_{m}$  is uniformly bounded for $0<R<1$ by Proposition \ref{pro:esti-PR}. Therefore,
 $\sum_{m=0}^{\infty} R^{m}\mathbb{P}_{m}\phi$ is also uniformly bounded for $0<R<1$ by using identity \eqref{def-con-pm-2}. Consequently,
\begin{equation*}\label{Pf--f}
\lim _{R \rightarrow 1^-} \sum_{m=0}^{\infty} R^{m}\mathbb{P}_{m}\phi =\phi
\end{equation*}
 holds for all $\phi \in L^p(\mathcal{N})$. \end{proof}

\begin{proof}[Proof of Corollary \ref{cor:f-decom}.] It follows from Theorem \ref{pf=f+pvf} that $\mathbb{P}_m$  is self adjoint since
\begin{equation*}\begin{aligned}
\langle  \psi,\,\, \phi\ast p.v.P_m \rangle =&\lim_{\varepsilon\rightarrow 0}\int_{\|(\mathbf{y},\mathbf{t})^{-1}(\mathbf{x},\mathbf{s})\|\geq\varepsilon} \psi(\mathbf{x},\mathbf{s})\overline{P_{m}((\mathbf{y},\mathbf{t})^{-1}(\mathbf{x},\mathbf{s}))\phi(\mathbf{y},\mathbf{t})}
d\mathbf{y}d\mathbf{t}d\mathbf{x}d\mathbf{s}\\
=&\lim_{\varepsilon\rightarrow 0}\int_{\|(\mathbf{y},\mathbf{t})^{-1}(\mathbf{x},\mathbf{s})\|\geq\varepsilon} P_{m}((\mathbf{x},\mathbf{s})^{-1}(\mathbf{y},\mathbf{t}))\psi(\mathbf{x},\mathbf{s})\overline{\phi(\mathbf{y},\mathbf{t})}
d\mathbf{y}d\mathbf{t}d\mathbf{x}d\mathbf{s}\\
=&\langle  \psi\ast p.v.P_m,\phi \rangle,
\end{aligned}\end{equation*}
for $\phi, \psi \in C_0^\infty(\mathcal{N})$, where the second identity holds since  the kernel $P_m$ satisfies \eqref{p-y-t=pyt}, and the constant $C$ in \eqref{eq:pf=f+pvf} is a real number, which can be calculated as \cite[P. 364-365]{S} in the case of the Heisenberg group.
Therefore,
\begin{equation}\label{pm-pm}
\langle \mathbb{P}_m \phi,\,\, \mathbb{P}_{m^{\prime}}\phi \rangle =\langle \mathbb{P}_{m^{\prime}} \mathbb{P}_{m} \phi,\,\, \phi\rangle
=\langle \delta_{m}^{m^{\prime}}\mathbb{P}_{m^{\prime}}\phi, \,\,\phi\rangle=\delta_{m}^{m^{\prime}}\|\mathbb{P}_{m^{\prime}} \phi\|_{L^2}^2,
\end{equation}
by Theorem \ref{thm:con-Pm-bounded}.  Hence,
\begin{equation*}\label{L2-inner}
\sum_{m=0}^{\infty}R^{2m}\left\|\mathbb{P}_m \phi\right\|_{L^2}^2=\left\|\sum_{m=0}^{\infty}R^m\mathbb{P}_m \phi\right\|_{L^2}^2 \longrightarrow \left\| \phi\right\|_{L^2}^2,
\end{equation*}
by Theorem \ref{thm:f-decom}. Consequently, we get
\begin{equation*}
\sum_{m=0}^{\infty}\left\|\mathbb{P}_m \phi\right\|_{L^2}^2 \geq \left\| \phi\right\|_{L^2}^2.
\end{equation*}

On the other hand,
\begin{equation*}
\sum_{m=0}^{N}\left\|\mathbb{P}_m \phi\right\|_{L^2}^2 =\left\|\sum_{m=0}^{N}\mathbb{P}_m \phi\right\|_{L^2}^2 \leq\left\| \phi\right\|_{L^2}^2,
\end{equation*}
by \eqref{pm-pm}, since $\sum_{m=0}^{N}\mathbb{P}_m$ is also a projection operator.  Letting $N\rightarrow \infty$, we get
\begin{equation*}
\sum_{m=0}^{\infty}\left\|\mathbb{P}_m \phi\right\|_{L^2}^2  \leq\left\| \phi\right\|_{L^2}^2.
\end{equation*}
Therefore we must have
\begin{equation*}
\sum_{m=0}^{\infty}\left\|\mathbb{P}_m \phi\right\|_{L^2}^2  =\left\| \phi\right\|_{L^2}^2.
\end{equation*}
Hence, \eqref{L2-decom} holds in $L^2$. The corollary is proved.
\end{proof}

\end{document}